\documentclass[11pt, oneside,a4paper]{article}
\usepackage[ansinew]{inputenc}
\usepackage[english]{babel}
\usepackage{amsbsy,amscd,amsfonts,amsmath,amsopn,amssymb,amstext,amsthm,amsxtra,array,color,fancybox,float,graphicx,latexsym}
\usepackage{mathrsfs,subfigure,srcltx,times,tikz,url,comment}
\usepackage{subfigure,srcltx,times,tikz,url,leftidx, caption}

\date{}

\newcommand{\mx}[1]{W_n}

\newcommand{\sucen}[2]{{#1_{1},\hdots,#1_{#2}}}

\newcommand{\N}{\mathbb{N}}
\newcommand{\Z}{\mathbb{Z}}

\newcommand{\R}{\mathbb{R}}

\newcommand{\BB}{\mathcal{B}}

\newcommand{\DD}{\mathcal{D}}

\newcommand{\GG}{\mathcal{G}}

\newcommand{\NN}{\mathcal{N}}

\newcommand{\PP}{\mathcal{P}}
\newcommand{\QQ}{\mathcal{Q}}

\newcommand{\VV}{\mathcal{V}}

\newcommand{\SSSS}{\mathcal{S}}

\newcommand{\DDD}{\mathscr{D}}

\newcommand{\PPP}{\mathscr{P}}

\newcommand{\id}{\operatorname{id}}

\renewcommand{\int}{\operatorname{int}}

\newcommand{\dd}{\operatorname{d}}

\newcommand{\row}{\operatorname{row}}

\newcommand{\col}{\operatorname{col}}

\newcommand{\tcone}{\operatorname{t--cone}}
\newcommand{\tconv}{\operatorname{t--conv}}

\newcommand{\diag}{\operatorname{diag}}

\newcommand{\sign}{\operatorname{sign}}

\newtheorem{thm}{Theorem}[section]
\newtheorem{lem}[thm]{Lemma}
\newtheorem{dfn}[thm]{Definition}
\newtheorem{cor}[thm]{Corollary}
\newtheorem{nota}[thm]{Notation}
\newtheorem{prop}[thm]{Proposition}
\newtheorem{ex}[thm]{Example}

\newtheorem{rem}[thm]{Remark}

\title{Quasi--Euclidean classification of Alcoved Convex Polyhedra\footnote{DOI: 10.1080/03081087.2019.1572065}}
\author{
M.J. de la Puente
\thanks{Partially supported by  Ministerio de Econom\'{\i}a y Competitividad, Proyecto I+D MTM2016-76808-P
and by UCM research group 910444.}
\\ Dpto. de \'{A}lgebra, Geometr\'{\i}a y Topolog\'{\i}a\\ Facultad de Matem\'{a}ticas\\ Universidad Complutense\\\texttt{mpuente@ucm.es}\\ Phone: 34--91--3944659
}

\begin{document}
\maketitle
\begin{abstract}

We give the quasi--Euclidean
classification of  the  maximal (with respect to the $f$--vector)  alcoved  polyhedra. The $f$--vector  of these
maximal  convex bodies  is $(20,30,12)$, so they are  simple dodecahedra.
We find eight  quasi--Euclidean  classes. This classification, which preserves angles, is finer than the known
combinatorial classification (found in 2012 by Jim\'{e}nez and de la Puente), which has only six classes.
Each alcoved polyhedron $\PP$ is represented by a unique visualized idempotent matrix $A$.
Some  2--minors of $A$ are invariants of $\PP$: they are  the tropical edge--lengths of $\PP$.


\end{abstract}

\section{Introduction}

\label{ALGEBRAIZATION}
Alcoved polyhedra in $\R^3$
make a rich, very well--behaved and beautiful class of convex bodies.
They have facet equations only of two types: $x_i=cnst$  and $x_i-x_j=cnst$. Two consequences follow. First, the $f$--vector,
the facet sequence, the facet angles and the dihedral angles are restricted. Second, the constant terms of the facet equations can
be assembled into a matrix, leading to an \textbf{algebraization  of alcoved polyhedra}.

Every  alcoved polyhedron  $\PP\subset \R^3$
can be represented by a normal idempotent (NI) order four matrix $A$. We write $\PP=\PP(A)$ to indicate this relationship.
To a certain translate  \label{com:translate} of $\PP$  there corresponds a unique matrix, which in addition to the former properties,
 is visualized (VI). In this paper we show that an alcoved polyhedron $\PP$ can be handled  by handling its corresponding matrix $A$.

\label{MAXIMAL}
The $f$--vector of a polyhedron  $\PP$ lists the number of vertices, edges and facets of $\PP$. For an alcoved one, it is known to be  bounded above by $(20,30,12)$.
\textbf{Maximal} alcoved polyhedra are those having such maximal $f$--vector. They are  simple (i.e., trivalent in vertices)
dodecahedra with 20 vertices and 30 edges.

\label{QUASI_EUCLIDEAN}
In this paper we classify alcoved polyhedra  from a topological,  affine and Euclidean point of view, i.e., we transform them by  small perturbations, affine maps and
space symmetries. Here is what we do.  First, we declare all boxes to be equivalent (disregarding length).
Having done so, we are left with the classification of alcoved polyhedra having a common bounding box. Say this common bounding box is
$\QQ$, the unit cube centered at the origin. Two alcoved polyhedra $\PP$ and $\PP'$ whose bounding box is $\QQ$ are equivalent for us,
in two cases:  if
there exists a symmetry of $\QQ$ taking $\PP$ bijectively onto $\PP'$ or
if $\PP'$ is a sufficiently small perturbation of $\PP$ (or a combination of both). Finally,
when we want to compare two arbitrary alcoved polyhedra $\PP$ and $\PP'$, first  we   transform their bounding boxes
$\BB$ and $\BB'$ into  $\QQ$  by bijective affine maps $f$ and $f'$. If
$f(\PP)$ and  $f'(\PP')$ are equivalent,as described above, then we define the original $\PP$ and $\PP'$ to be equivalent.
 It is an   angle--preserving equivalence  relation  between maximal alcoved polyhedra.
We have composed the
expression  \textbf{quasi--Euclidean} equivalence to refer to  such a relation.

\label{MAIN_THM}
Trivially, in $\R^2$ the maximal alcoved polygons  are hexagons (slopes of sides being $0,1,+\infty $) and  there exists just one quasi--Euclidean class of such. In $\R^3$ the maximal alcoved polyhedra  are dodecahedra. Our main  theorem  goes as follows \label{com:maximal}

\begin{thm}\label{thm:eight_classes}
The quasi--Euclidean classification in the family of maximal alcoved dodecahedra in $\R^3$ has eight classes.
\end{thm}

\label{NON--MAXIMAL ALCOVED}
Every non--maximal alcoved polyhedron arises from a maximal one, when  some edges collapse,  and simplicity is lost. Instances of such
are boxes.
With similar arguments, the  quasi--Euclidean classification of  non--maximal alcoved polyhedra is possible to devise, but we believe
it is far less interesting (and far more involved).  \label{com:non_maximal}

\medskip
\label{IDEA OF PROOF}
The  proof of the main theorem  is completed in
p.~\pageref{proof:main_thm}. It is reached after the following steps.
In each alcoved polyhedron $\PP$,
two distinguished vertices are called \textbf{North  and South  Poles}, and marked $\NN$  and $\SSSS$. After  an idea of Kepler's,
\textbf{the polyhedron $\PP$ is the union of three parts:  North Cask, South Cask and Equatorial Belt}.
For instance, each box has a trivial Equatorial Belt. Each Polar Cask has a \textbf{Cask type}. Cask types are described by
a \textbf{vector} and a \textbf{chirality word}.  In the maximal case, the vector is $(p.q.r)$, with $p,q,r\in\{4,5,6\}$ and $p+q+r=15$.
The chirality word is either \textbf{left} or \textbf{right} and it is omitted whenever it can be deduced.
In the maximal case, we show that  both the Equatorial Belt and the quasi--Euclidean class are determined only by the North and South
Cask types. The Cask types are determined only by the signs of six  2--minors of the  NI  matrix $A$ representing $\PP$.
Ultimately, the Cask types  are determined by the signs of a 6--tuple, called \textbf{difference tuple}, which is computed from the
perturbation matrix $E$ of $A$ in a very direct way (see example \ref{ex:example}).

\medskip
\label{CLASSES}
The classes are denoted  $QE1,QE2,\ldots,QE8$ and the classification is  shown in table \ref{table:eight_classes}, which reads
as follows.
The rectangle on the first row left describes a representative  of class $QE1$: it is a dodecahedron having a North Cask type vector $(4.5.6)$ and a South Cask type  vector   $(4.5.6)$.
This means that  the North Cask is composed of  the facet $x_1=cnst$, which is a 4--gon, the facet $x_2=cnst$ a 5--gon, and the
facet $x_3=cnst$ a 6--gon and they all meet at $\NN$.  The chirality word here can be deduced from the former data and has
been omitted. It is left. In the South Cask there is  a facet $x_1=cnst$, which is a 4--gon, a facet $x_2=cnst$ a 5--gon, and a
facet $x_3=cnst$ a 6--gon. Similarly for other rectangles in the table.

\begin{table}[ht]
\begin{center}
\begin{tabular}{| >{$}l<{$}| >{$}l<{$}| >{$}l<{$}| >{$}l<{$}| >{$}l<{$}|}
\hline
\NN (4.5.6)&\NN (4.5.6)&\NN (4.5.6)&\NN (4.5.6)&\NN (4.5.6)\\
\SSSS (4.5.6)&\SSSS (5.6.4)&\SSSS (6.5.4)&\SSSS (4.6.5)&\SSSS (5.4.6)\\
QE 1&QE 2&QE 3&QE 4&QE 5\\
\hline
&\NN (5.5.5)\ \text{right}&\NN (4.5.6)\ \text{left}&\NN (5.4.6)\ \text{right}&\\
&\SSSS (5.5.5)\ \text{left}&\SSSS (5.5.5)\ \text{left}&\SSSS (5.5.5)\ \text{left}&\\
&QE 6&QE 7&QE 8& \\
\hline
\end{tabular}
\end{center}
\caption{The eight quasi--Euclidean classes of maximal alcoved dodecahedra in $\R^3$. In each case,  the third line shows the class name.
The first and second  lines describe the North and South Cask types of one  representative of the class.}
\label{table:eight_classes}
\end{table}

\label{CHIRALITY}
Alcoved polyhedra may have pentagonal facets and, as soon as pentagons appear, \textbf{chirality} appears too. Chirality
changes the orientation.
A chiral copy of $\PP(A)$  belongs to the same quasi--Euclidean
class. Algebraically, we
produce a chiral copy of $\PP(A)$ by letting some group element act on  the matrix $A$.

\medskip

\label{KEY FACT}
Our starting point is the known fact that a  \textbf{normal idempotent of order $n$ matrix $A$  provides an  alcoved polytope  $\PP(A)$ in
$\R^{n-1}$},  idempotency with respect to \textbf{tropical multiplication}. Translations of $\PP(A)$ correspond to conjugations of $A$
by diagonal matrices. It is particularly nice when $\PP$ is \textbf{cornered}, meaning that the largest point in $\PP=\PP(A)$ is
the origin. In such a case, the corresponding matrix  $A$ is \textbf{visualized idempotent} (VI) matrix, a very nice matrix to
compute with.
\medskip

\label{2--MINORS AND NORTH AND SOUTH CASKS}
Consider $n=4$. For a VI matrix $A$, we prove that the \textbf{tropical edge--lengths} of $\PP(A)$  are equal to either  the
entries  of the box matrix $B$ or the entries of the
perturbation matrix $E$, where $A$ is decomposed as the sum $A=B-E$ (classical sum here). This algebraization of the edge--lengths
of $\PP(A)$ is a key fact.
We prove that all these entries (of $B$ and of $E$) are 2--minors of $A=B-E$. The 2--minors are certain sums and differences of
matrix entries.

We prove that 2--minors are preserved by visualization.
This means that the edge--lengths of $\PP(A)$ are certain 2--minors of $A$, when $A$ is only NI (but perhaps not VI, and $A$ does not admit a box--perturbation decomposition). In other words,
\emph{certain 2--minors of $A$  are metric invariants of $\PP(A)$}, for any NI order 4 matrix $A$.

For $A=B-E$, the bounding box of $\PP(A)$ is just $\PP(B)$ and the \textbf{cant} and \textbf{difference tuples} $c$  and $d\in\R^6$
are  obtained easily
from the matrix $E$. Both $B$ and $c$ carry the metric information of the polyhedron $\PP(A)$. The difference tuple $d$ carries
the quasi--Euclidean class information of $\PP(A)$.
With these tools, we first study the North Cask of $\PP(A)$. Then, we study the South Cask of $\PP(A)$, by turning $\PP(A)$ around
by a Polar Exchange, which amounts to
transforming the matrix $A$ under the action of some group element.

\label{QUASI_EUCLIDEAN AGAIN}
Our classification is both geometrical (in the sense of Euclidean geometry) and topological.
Indeed, for polyhedra whose bounding
box is the unit  cube $\QQ$, we
say that  two alcoved polyhedron $\PP$, $\PP'$ are equivalent whenever either a cube symmetry takes $\PP$ onto $\PP'$ or when $\PP'$ is
a small perturbation of $\PP$, or both. The \textbf{group} of cube
 \textbf{symmetries} relevant here, denoted $\GG_3$, is isomorphic to  $\Z_2\times \Sigma_3$. Here $\Sigma_3$ denotes the
permutation group in $3$ symbols.

\medskip

The  ideas of this paper have been applied to exact volume computation in \cite{Puente_Claveria}.
Our techniques might admit a  generalization to alcoved polytopes in higher dimensions, but it is certainly not straightforward.
Thus  we have a sound reason to work today with alcoved polyhedra in 3--dimensional space. Other more general alcoved polytopes arise
from root systems.
There is a connection  between alcoved polyhedra and the root system $A_3$. Indeed, the facet exterior normal vectors of an alcoved polyhedron are $\{\pm u_i:i\in[3]\}\cup\{\pm (u_i-u_j): i\neq j\in[3]\}$, a proper subset of the  $A_3$ roots, which are $\{\pm u_i:i\in[3]\}\cup\{\pm u_i\pm u_j: i\neq j\in[3]\}$. Here  $(u_1,u_2,u_3)$ denotes the canonical basis in $\R^3$.

\medskip
Alcoved polytopes have been studied in great generality  since 2007, from a combinatorial point of view;
see \cite{Lam_Postnikov,Lam_PostnikovII,Postnikov_per}.
In relation with
Tropical Mathematics,  they appeared earlier in connection with max--plus spectral theory and matrix eigenspaces
in \cite{Akian_HB,Akian_Gaubert,Akian_Gau_Walsh} (see also the bibliography in the survey \cite{Akian_HB}). In relation
to classical convex polytopes and polytopal complexes,  they  have been studied  in the foundational paper \cite{Develin_Sturm}, and subsequent papers \cite{Jimenez_Puente,Joswig_Kulas,Joswig_Sturm_Yu,Puente_kleene,Tran,Werner_Yu}.
In the recent  book \cite{Bump_Sch}, it is shown that alcoved polytopes  have a crystal base  (or Kashiwara crystal) structure.
In relation with  Jacobian curves, alcoved polytopes have been used  in \cite{Zharkov}, in order to prove a classical result on
algebraic curves
by means of tropical techniques.

\medskip
\textbf{Normal matrices} (and slightly weaker notions) have been studied for more than fifty years
(see the pioneer works \cite{Afriat,Blyth,Yoeli}),
under various different names.
In connection with
Tropical Mathematics, they appear in \cite{Butkovic,Butkovic_Libro,Butkovic_S,Butkovic_Sch_Ser,Puente_kleene,Sergeev} and very recently in \cite{Yu_Zhao_Zeng}.
In connection with Automata Theory, Scheduling Theory, Computer Science, Discrete Event Systems and  Dynamical Programming, normal matrices  appear in \cite{Belman,Berthomieu,Dima,Mine} among other, where  they are  named difference bounded matrices (DBM). Visualization of matrices is one way to standardize \textbf{normal idempotent matrices}, appearing
in \cite{Akian_HB,Sergeev_Sch_But}. Tropical Mathematics is
an emerging area with multiple connections to Analysis, Algebraic Geometry, Representations of Lie Algebras, Mathematical Economics
and Algebra:
see \cite{Brugalle_fran,Brugalle_engl,Franceses,Rusos,Litvinov_ed,Litvinov,Litvinov_ed_2,Mikhalkin_T,Richter,Speyer_Sturm_2}
and bibliographies therein.

\medskip
The paper is organized as follows. In section \ref{sec:labels} we introduce  left and right pentagons, North and South Poles, and
some combinatorial labels of  vertices of an alcoved polyhedron. In section \ref{sec:box_and_cube}
we present  boxes and cubes, together with their matrices. These are   the easiest alcoved polyhedra and the easiest matrices. More
general matrices are introduced in
section \ref{sec:alcoved_poly_and_mat}. In this section we gather the \textbf{core facts}:  we restrict ourselves to working with
normal idempotent matrices (NI)  because they provide alcoved polytopes (see the pioneer paper \cite{Develin_Sturm} on tropical
convexity vs. classical convexity).   Even more, if we
restrict ourselves
to visualized idempotent matrices $A$ (VI for short), then
the edge--lengths of the corresponding alcoved polyhedron (which is cornered) can easily be read from the decomposition $A=B-E$
of the matrix. In a subsection, we introduce the useful notion of 2--minors of a matrix $A$. Some of them are edge--lengths of $\PP(A)$.
In section \ref{sec:from_box_to_poly} we introduce bounding boxes, perturbations and cant and difference tuples. More
precisely,  \emph{every alcoved polytope is the perturbation of  a box, and every box is an unperturbed alcoved polyhedron}.
The perturbation consists in canting (i.e., beveling) a cycle of  box edges. In section \ref{sec:symmetries} we introduce the
subgroup, denoted $\GG_3$, of cube symmetries fixing a given diagonal.
This group acts on the set $M_4^{NI}$ of normal idempotent matrices of order four.
In the short sections \ref{sec:maximality} and \ref{sec:dist} we make precise the notions of maximality and of tropical distance.
Section \ref{sec:casks_and_belt} contains the heavy work: we define North and  South  Casks  and we show that Cask types of a maximal
polyhedron $\PP(A)$ depend on the signs of certain 2--minors of $A$. Ultimately, the Cask types depend on the signs of the difference
tuple $d$ of $A$. The quasi--Euclidean classification is presented
in section \ref{sec:qe_classif}. There is an \textbf{action of a  group} isomorphic to $\GG_3$ on a set of sign 6--tuples. The \textbf{orbit} of
$\sign(d)$, the sign
of the difference tuple $d$ of $A$, provides the equivalence class of $\PP(A)$. A worked  example is given in
p.~\pageref{ex:example}.

\medskip
The relationship of the present classification with the coarser combinatorial classification (found in 2012 by Jim\'{e}nez and de la Puente \cite{Jimenez_Puente}) is shown in
table \ref{table:qe_vs_combinatorial}. Positive answers (in the realm of alcoved polyhedra) to two open questions are given
in section \ref{sed:final}.

The number 8 plays two roles in this paper: on the one hand, it is the number of North and South Cask Types and, on the other hand, it is  the number of quasi--Euclidean clases.

\medskip

\textbf{Notations and colors:} in figures, the origin is  marked with a red dot and generators are marked
with blue dots, three--generated vertices are marked with magenta dots, two--generated vertices are
marked with yellow dots.
The unit cube centered at the origin is denoted $\QQ\subset \R^3$. Vectors and points will be  written in columns.
Given $n\in \N$, the set $\{1,2,\ldots,n\}$ is denoted $[n]$. Let $\Sigma_n$ denote the
permutation groups in $n$ symbols.

\begin{figure}[ht]
\centering
\includegraphics[width=6cm]{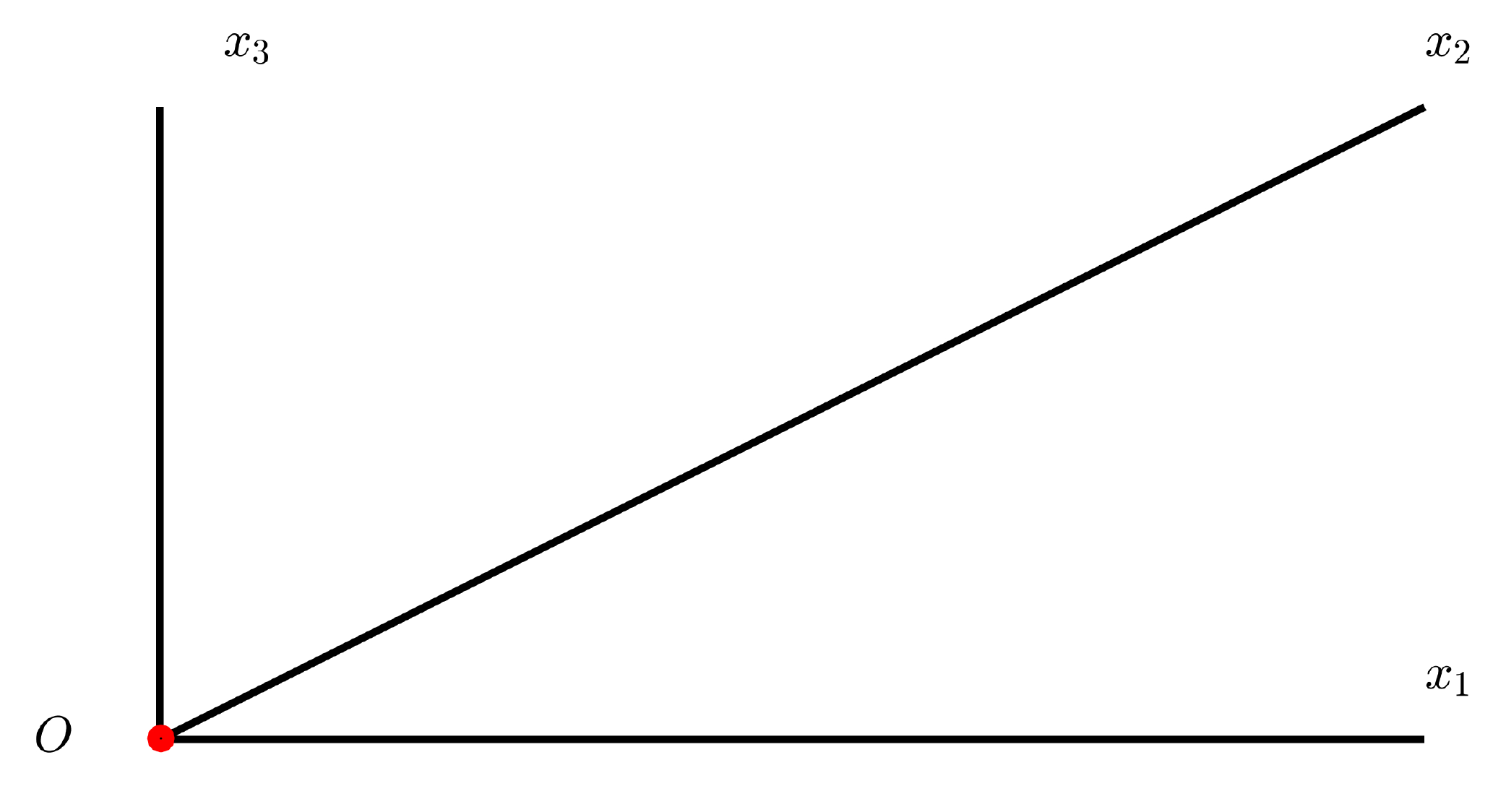}\qquad\includegraphics[keepaspectratio,width=4cm]{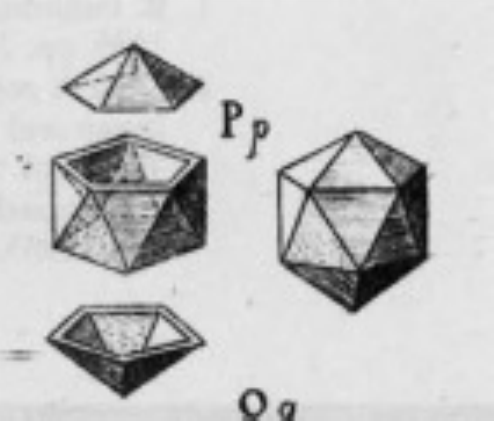}\\
\caption{Coordinate axes in $\R^3$;  the red dot marks the origin (left). North and South Casks and Equatorial Belt in an icosahedron in
figure from  Kepler's  \emph{Harmonices Mundi}, 1691 (right).}\label{fig_01_R3_Kepler}
\end{figure}

\section{Left and right pentagons, geographical terms and some combinatorial labels}\label{sec:labels}

\begin{dfn}[Left pentagon and right pentagon]\label{dfn:left_penta}
Consider a rectangle $F\subset \R^2$ with edges parallel to the coordinate axes  and  fix a vertex $\VV$ of $F$ such that $\VV=\max R$
or $\VV=\min R$, coordinatewise.
Cut off  a corner of $F$ not opposite to $\VV$ and obtain a pentagon $F'$.
We say that $F'$ is a \textbf{left  pentagon} if the removed corner is on the left hand side of $F$. Otherwise we say $F'$
is a \textbf{right  pentagon}.
\end{dfn}

Two distinguished vertices exist in  every alcoved polyhedron $\PP\subset \R^3$: the \textbf{North  and South Poles}. These are  $\NN=\max \PP$ and
$\SSSS=\min \PP$, where  maxima and minima are computed coordinatewise. The classical line  passing through $\NN$ and $\SSSS$ will be
called the \textbf{Polar Axis} of $\PP$. The largest point in $\PP$ is the North Pole. The smallest point in $\PP$ is the South Pole.

\begin{figure}[ht]
\centering
\includegraphics[width=5cm]{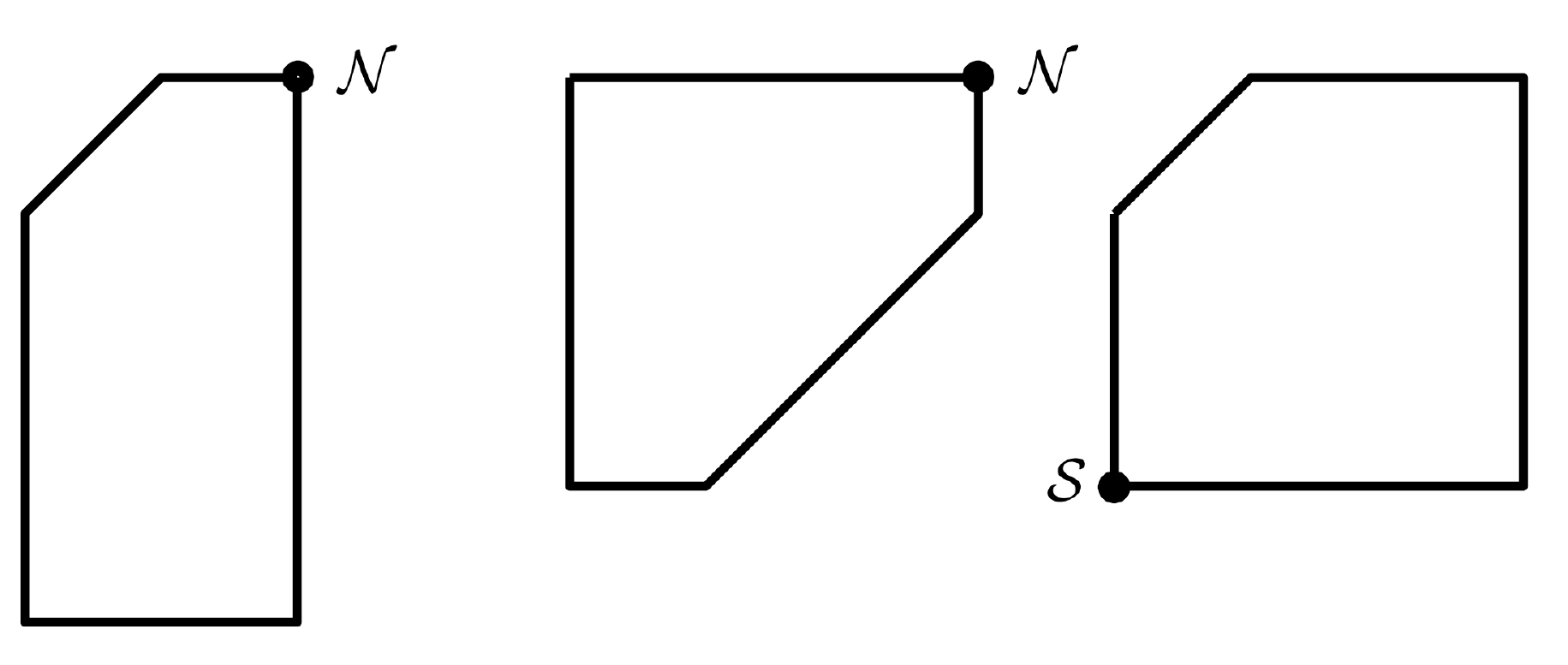}
\caption{A left pentagon at the North Pole is such that  the left corner of a rectangle  is missing (left).
A right pentagon at the North Pole: the right corner has been cut off (center). A left pentagon at the South Pole: the left corner has
been cut off (right).}
\label{fig_02_pentagon_left_and_right}
\end{figure}

In addition, we  use a \textbf{combinatorial labeling} \label{dfn:labels} of vertices, marked with underlined digits.
The single--digit vertices of $\PP$ are called \textbf{generators} (marked blue in figures); there are four of them:
$\underline{j}$, for $j\in [4]$.

The labels of the remaining
vertices of $\PP$ may have $2$ or  $3$ digits.
A vertex in $\PP$ may admit two or more labels and this  happens if and only
if $\PP$ is \textbf{non--maximal} (with respect to $f$--vector). \label{com:maximal_2} More combinatorial labels are introduced in p.~\pageref{dfn:more_labels}.

\section{Boxes and Cubes}\label{sec:box_and_cube}
 We identify $\R^{n-1}$ with the hyperplane $\{x_n=0\}$ in $\R^n$.

\begin{dfn}\label{dfn:box}
A \textbf{box} in $\R^{n-1}$ is a polytope $\BB$
whose facets are contained in hyperplanes parallel to the
coordinate axes.
A \textbf{cube} is a box of equal edge--lengths.
\end{dfn}

\begin{figure}[ht]
 \centering
  \includegraphics[width=12cm]{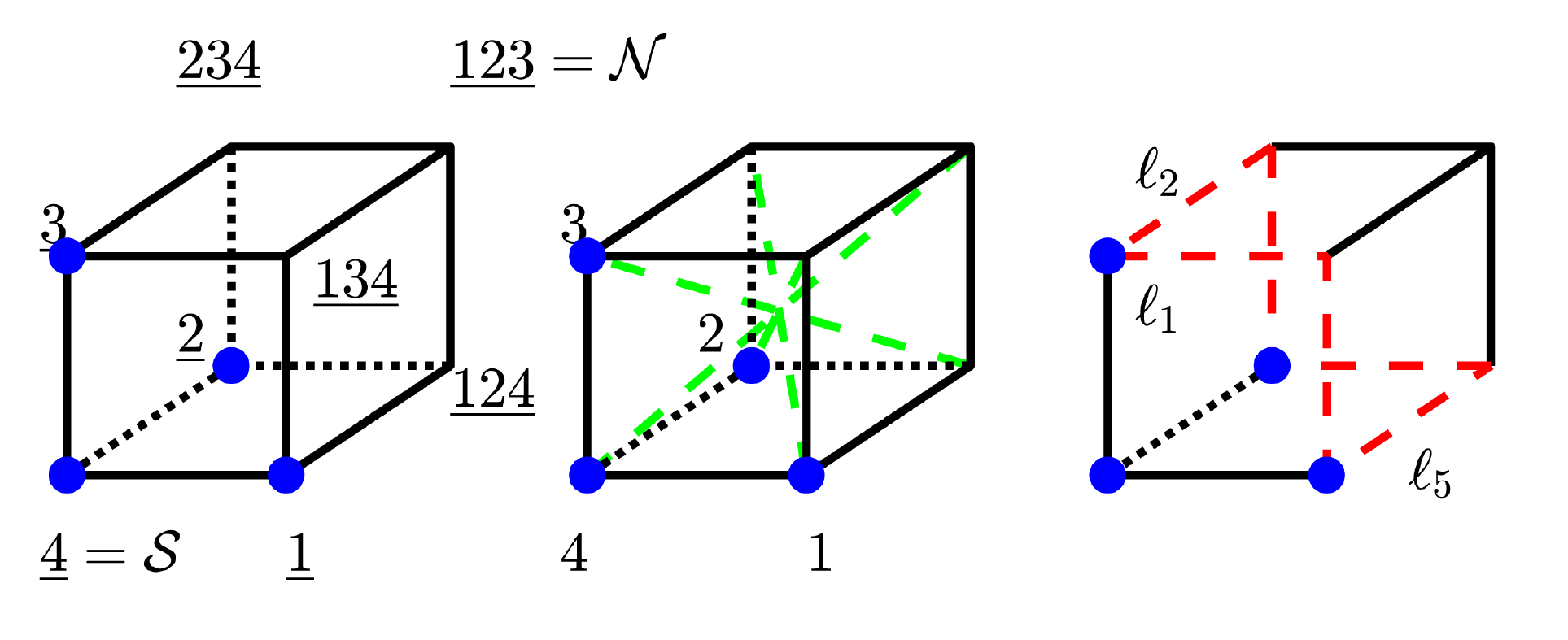}\\
  \caption{In $\R^3$, labeled  vertices of a cube, using underlined single--digit or underlined three--digit labels (left).
  Labeled diagonals of a cube (diagonals marked in green dashed lines) (center).
  Cantable edges $\sucen{\ell}6$ in  a cube (right). They form a cycle, beginning with $\ell_1$ in front top, $\ell_2$ top left, $\ell_3$  left back, $\ell_4$ back bottom, $\ell_5$  bottom right and  $\ell_6$ right front.
   Cantable edges marked in red dashed lines.
   }\label{fig_03_three_cubes}
\end{figure}

\begin{dfn}[Labels in a box]\label{dfn:labels_box}
\begin{enumerate}
\item The vertices of a 3--dimensional box are labeled $\underline{j}$, with $j\in[4]$ and $\underline{123},
\underline{124}, \underline{134}, \underline{234}$.
In particular, $\underline{123}$ is the  \textbf{North Pole} and $\underline{4}$ is the  \textbf{South Pole}.
\item We label the diagonals of boxes  with plain (not underlined) digits as follows: the diagonal $j$
contains the generator $\underline{j}$, for $j\in[4]$ (see
figure \ref{fig_03_three_cubes}, center). In particular, the label of the \textbf{Polar Axis} is 4. \label{item:Polar_Axis}
\end{enumerate}
\end{dfn}
For $j\in [3]$ consider the facet $F_j$  of a box all whose vertices contain $j$ in the label.
Then, it is easy to see that $\underline{j}$ is the smallest point of $F_j$.
\begin{nota}\label{nota:box_matrix}
Given a vector $t=(t_1,t_2,t_3)^T\in\R^3_{\le0}$, consider the
following matrix
\begin{equation}\label{eqn:box_matrix_1}
B(t):=\left(\begin{array}{rrrr}0&t_1&t_1&t_1\\
t_2&0&t_2&t_2\\
 t_3&t_3&0&t_3\end{array}\right).
\end{equation}
Since $\R^3$ is identified with
$\{x_4=0\}$ in $\R^4$,
then  the matrix $B(t)$ is identified with the matrix
\begin{equation}\label{eqn:box_matrix_2}
\left(\begin{array}{rrrr}0&t_1&t_1&t_1\\
t_2&0&t_2&t_2\\
 t_3&t_3&0&t_3\\0&0&0&0\end{array}\right).
\end{equation}
It is easy to check that
the columns of $B(t)$ are the generators of a box
$\BB\subset \R^3$ whose North Pole is the origin and whose edge--lengths are $|t_1|,|t_2|,|t_3|$.
We say that the matrix $B(t)$ represents the box $\BB$.
\end{nota}

\section{Alcoved convex polyhedra and their matrices}\label{sec:alcoved_poly_and_mat}
In the previous section we have seen that certain  matrices represent boxes. Here we will see matrices which represent alcoved polyhedra.
Let $x_1,\ldots,x_{n-1}$ be coordinates in $\R^{n-1}$.
\begin{dfn}[Alcoved]\label{dfn:alcoved_convex}
A convex polytope $\PP$  in $\R^{n-1}$ is \textbf{alcoved} if the facets of $\PP$
are contained in  hyperplanes of equations $x_i=cnst$ or $x_i-x_j=cnst$,  for $i,j\in[n-1]$.
\end{dfn}
In particular, a box is a \textbf{non--maximal} (with respect to $f$--vector)
alcoved polytope.\label{com:non_maximal_2}

\begin{dfn}[Cornered alcoved polytope]\label{dfn:cornered_p}
An alcoved polytope   $\PP\subset \R^{n-1}$ is \textbf{cornered}
if  the North Pole $\NN$ is at the origin. In particular, $\PP$ is contained in the closed non--positive orthant of $\R^{n-1}$.
\end{dfn}

A way to assemble  all the constant terms of the linear equations describing  an alcoved polyhedron $\PP$  is to arrange them into a
\textbf{matrix} $A=(a_{ij})$  as follows:
\begin{eqnarray}
a_{nk} \le x_k\le -a_{kn},\quad  a_{ij}\le x_i-x_j\le -a_{ji}, \quad i,j,k\in [n-1],\\\label{eqn:alcoved_2}
a_{ii}=0, \qquad i\in[n].\label{eqn:alcoved_3}
\end{eqnarray}

Let $M_n(\R_{\le0})$ denote the set of square non--positive real matrices of order $n$.

\begin{dfn}[Normal idempotent matrix]\label{dfn:normal}
Consider $A=(a_{ij})\in M_n(\R_{\le0})$.

\begin{enumerate}
\item The matrix $A$ is called \textbf{normal} if
$a_{ii}=0$,  all $i\in [n]$, i.e., $A$ is is zero--diagonal.\label{item:N}
\item The normal matrix $A$ is called \textbf{tropically                                                                     idempotent} if $a_{ik}+a_{kj}\le a_{ij}\le0$, all $i,j,k\in[n]$.\label{item:NI}
\end{enumerate}
\end{dfn}
\begin{rem}
Let us use \textbf{tropical algebra} over the extended real numbers $\R\cup\{-\infty\}$ to operate with numbers and $n$--dimensional  vectors.  The tropical operations
are $\oplus:=\max$  (tropical addition) and  $\odot:=+$ (tropical multiplication), coordinatewise. The element $-\infty$ is  neutral
for tropical addition.\footnote{We will use the alien element $-\infty$ as little as possible.}
For example: $(2,3)\oplus(-1,5)=(2,5)$, $(-4)\odot(2,-\infty)=(-2,-\infty)$ and  $(2,3)\odot(-1,5)^T=\max\{1,8\}=8$.
Now extend these tropical operations to matrices, following the  usual way.
In this context, it is well--known that any normal matrix satisfies $A\le A\odot A$. Therefore, a \textbf{normal matrix} $A$
is \textbf{idempotent}
if  and only
if  $A\odot A\le A$, which (since tropical addition means computing a maximum) is expressed  as $a_{ik}+a_{kj}\le a_{ij}$, as shown in
item \ref{item:NI} above.

Tropical addition  $\oplus=\max$ is \textbf{idempotent}, since $a\oplus a=a$, all $a$. Given $a\in \R$, no opposite element exists for $a$ with the operation $\oplus$.

\end{rem}

\begin{rem}
We also use classical addition of matrices, when we talk about perturbations.
\end{rem}

\begin{dfn}
\begin{enumerate}
\item The \textbf{tropical
identity matrix} is the square matrix $I=(i_{kl})$ with $i_{kk}=0$, $i_{kl}=-\infty$, otherwise.
\item For $k,l\in [n]$, the \textbf{tropical
permutation matrix} is the square matrix $P^{(kl)}$   obtained from  $I$ by permuting the $k$--th and $l$--th rows.
\end{enumerate}
\end{dfn}

Tropically, the only \textbf{invertible} matrices are the diagonal ones: $D=\diag(b)$, with $b\in \R^n$ and
$-\infty$ outside the main diagonal. Indeed,  $D^{-1}=\diag(-b)$ is the topical inverse. The \textbf{conjugate} of a matrix $A$ by $D$ is
\begin{equation}\label{eqn:conj}
\leftidx{^D}{A}:=D\odot A\odot D^{-1}.
\end{equation}
\begin{nota}
The notation $\leftidx{^{P^{(kl)}}}A$ is simplified to $\leftidx{^{(kl)}}A$.
\end{nota}

Since any matrix can be normalized (however, not uniquely),
working with normal matrices may be advantageous,  and it is   no serious restriction. We will do so in the sequel.

\begin{rem}[CORE BACKGROUND]
The following are well--known facts.
If $A$ is a real matrix, then the set
\begin{eqnarray}\label{eqn:span}
\tconv(A):=\{x=\lambda_1\odot\col(A,1)\oplus \cdots\oplus\lambda_n\odot\col(A,n): \lambda_1,\ldots,\lambda_n\in \R\}\\
=\{x=A\odot \lambda: \lambda\in \R^n\}
\end{eqnarray}
of all tropical linear combinations of
the columns of $A$ is a \textbf{polytopal complex} in $\R^n$ of impure dimension  (see \cite{Develin_Sturm,Richter}).
We have $x=(x_i)$ with
\begin{equation}
x_i=\max\limits_{k\in[n]}\lambda_k+a_{ik}, \qquad i\in[n].
\end{equation}
Further, \textbf{if $A$ is normal idempotent then  $\tconv(A)$ is
much more than a complex: it is just one alcoved polytope}
(see \cite{Butkovic_Sch_Ser,Jimenez_Puente,Puente_kleene,Sergeev,Sergeev_Sch_But}). This is a \textbf{key fact}.
In such a case, the intersection of $\tconv(A)$ with the classical hyperplane $\{x_n=0\}$ is also  an \textbf{alcoved polytope} and
its dimension is one unit smaller.
It is denoted $\PP(A)$ in these notes:
\begin{equation}\label{eqn:span2}
\PP(A):=\{x\in\R^{n}:x_n=0\}\cap \tconv(A).\footnote{Instead of $\tconv(A)$, we could use as
well $\tcone(A)$, the definition being analogous to $\tconv(A)$, except that $\lambda_j$ runs in $\R\cup\{-\infty\}$. The set
$\tcone(A)$ is a tropical cone, and the difference between both sets is just the point $(-\infty,\ldots,\-\infty)^T$.
We get $\PP(A)=\{x\in\R^{n}:x_n=0\}\cap\tcone(A)$.}
\end{equation}
Conversely, $\PP(A)$   determines $\tconv(A)$, because $\tconv(A)$ is closed under classical addition  of the
constant vector $r(1,\dots,1)^T$ (i.e., closed
under tropical  scalar multiplication by $r$),   for all $r\in \R$.

In order  to \textbf{sketch} the polytope  $\PP(A)$ in $\R^{n-1}=\{x_n=0\}$,
we must compute  the following auxiliary matrix (called \textbf{geometric matrix})
\begin{equation}\label{eqn:geometric}
A_0=(\alpha_{ij}),\qquad \alpha_{ij}:=a_{ij}-a_{nj}.
\end{equation}
In particular, the last row in $A_0$ is zero.
Since each column of $A_0$ is a tropical scalar multiple of the corresponding column in $A$, we have  $\tconv(A)=\tconv(A_0)$.
The columns of $A_0$ belong to $\R^{n-1}$  and
\begin{equation}\label{eqn:equality_of_polys}
\PP(A)=\PP(A_0).
\end{equation}
If the matrix $A$ is normal, then both $\tconv(A)$  and  $\PP(A)$ contain the origin.\footnote{Indeed, for each $i\in[n]$, take $\lambda_i=0$, $\lambda_j<0$ and use $a_{ii}=0$ and $a_{ij}\le0$, to obtain $x_i=\max\limits_{k\in[n]}\{\lambda_k+a_{ik}\}=0$.}

A \textbf{converse holds}  (see \cite{Puente_kleene,Sergeev}): if $\PP\subset \R^{n-1}$ is an alcoved polytope containing the origin, then
there exists a normal idempotent matrix $A\in M_n^{NI}$ such that
\begin{equation}\label{eqn:equality}
\PP=\PP(A).
\end{equation}
Moreover, if the constant coefficients come from a normal idempotent matrix, then the inequalities  are tight (see
\cite{Afriat,Puente_kleene}).
\end{rem}

$M^{NI}_n$ is the set of
\textbf{normal idempotent} matrices of order $n$.
In view of (\ref{eqn:equality}), we will always work in  $M_n^{NI}$.

\begin{dfn}\label{dfn:visualized_m}
A normal matrix  $A$ is \textbf{visualized} if $A=A_0$. Equivalently, $a_{ii}=a_{ni}=0$ and $a_{ij}\le0$, all $i,j\in[n]$.
\end{dfn}

\begin{rem}
The union of the regions \label{rem:regions}
$$R_j:=\{x\in \mathbb{R}^{n-1}: 0\le x_j \text{\ and \ } x_k\le x_j,\ k\in[n-1]\},\quad j\in[n-1],$$
$$R_n:=\{x\in \mathbb{R}^{n-1}: x_k\le 0,\ k\in[n-1]\}$$ is the whole $\R^{n-1}$.
\end{rem}

We are dealing with   zero--diagonal matrices  $A$ as well as matrices with zero last row $A_0$, which are  are related
by (\ref{eqn:geometric}) and (\ref{eqn:equality_of_polys}).
The geometric matrix $A_0$ is used to graph the polytope of $\PP(A)\subset \R^{n-1}$.
The following is a characterization of a normal matrix and of a visualized matrix $A$ in terms of the location of the columns of
the  geometric matrix $A_0$.\footnote{It implies that, for a NI matrix,  vertex  labels in the North Cask of $\PP$ follow the same cyclic
sequence for every polyhedron, when  going around $\NN$. This will be used in p.~\pageref{com:geom_interpret}.}

\begin{lem}[Location of columns] \label{lem:geom_interpret}
For a zero--diagonal matrix $A$, the following hold:
\begin{enumerate}
\item $A$ is normal if and only if $\col(A_0,j)\in R_j$, all $j\in[n]$,
\item $A$ is visualized  if and only if $\col(A_0,j)\in R_j\cap R_n$, all $j\in[n]$.
\end{enumerate}
\end{lem}
\begin{proof}
\begin{enumerate}
\item $(\Rightarrow)$ Fix $j\in [n]$. The $i$--th coordinate of $\col(A_0,j)$ is $x_i=\alpha_{ij}=a_{ij}-a_{nj}$ and $j$--th
coordinate is $x_j=a_{jj}-a_{nj}$. We have $x_i\le x_j$ and $0\le x_j$ because $a_{jj}=0$ and the entries of $A$ are
non--positive. Thus $\col(A_0,j)\in R_j$. The converse is easy.
\item $(\Rightarrow)$ By the item above, we have $\col(A_0,j)\in R_j$, all $j\in[n]$. The $i$--th coordinate of
$\col(A_0,j)=\col(A,j)$ is $x_i=a_{ij}\le0$, whence $\col(A_0,j)\in R_n$, all $j\in[n]$.  The converse is easy.
\end{enumerate}
\end{proof}

The set of \textbf{visualized} matrices of order $n$ is a real cone of dimension $(n-1)^2$.
The subset of \textbf{visualized idempotent} matrices in denoted $M^{VI}_n$ and we have  $M_n^{VI}\subset M_n^{NI}$.

Any normal idempotent  matrix can be visualized, as the
following  lemma  shows.

\begin{lem}[Visualizing a normal idempotent matrix] \label{lem:visualizing}
Given $A\in M_n^{NI}$, the matrix $\leftidx{^D}{A}$ is visualized idempotent, where
$D=\diag(\row(A,n))$.
\end{lem}
\begin{proof}
If $A$ is idempotent, then   so is $D\odot A\odot D^{-1}$.
The diagonal of $\leftidx{^D}{A}$ is  null, because the diagonal of $A$ is.
The  $(i,j)$--th entry  is
\begin{equation}\label{eqn:ij_visualization}
(\leftidx{^D}{A})_{ij}:=a_{ni}+a_{ij}-a_{nj}
\end{equation} and it is non--positive, by normality of $A$. This
shows normality of $\leftidx{^D}{A}$. The  $(n,j)$--th entry is $(\leftidx{^D}{A})_{nj}=a_{nn}+a_{nj}-a_{nj}=0$, since $a_{nn}=0$.
This proves that $\leftidx{^D}{A}$ is visualized.
\end{proof}
Visualization is a particular case of conjugation.
In  \cite{Sergeev_Sch_But} this process is called  \textbf{visualization scaling}.

\begin{rem}
The visualization of a normal matrix may not be normal. For example,
$B=\left( \begin{array}{rrr}
0&0&0\\-1&0&0\\-1&-2&0\\
\end{array}\right)
$,
$B\odot B=\left(\begin{array}{rrr}
0&0&0\\-1&0&0\\-1&-1&0\\
\end{array} \right)
$,
$B_0=\left( \begin{array}{rrr}
1&2&0\\0&2&0\\0&0&0\\
\end{array}\right)
$. The matrix
$\leftidx{^D}B=\left(\begin{array}{rrr}
0&1&-1\\-2&0&-2\\0&0&0\\
\end{array} \right)
$ is not normal.
\end{rem}

The following statements are  easy to prove.
For every alcoved polytope $\PP\subset \R^{n-1}$, there exists a unique vector $v\in \R^{n-1}$ ($v$ depending on $\PP$) such
that the  translated polytope $t_v(\PP)$ is cornered.
If $A\in M^{VI}_n$, then $\PP(A)$ is  cornered. The converse holds.

\begin{cor}[Unique representation]\label{cor:unique_rep}
For every cornered alcoved polytope   $\PP\subset \R^{n-1}$ there existe a unique matrix $A\in M^{VI}_n$ such that  $\PP=\PP(A)$.\qed
\end{cor}

\begin{dfn}[More labels]\label{dfn:more_labels}
Let $A\in M_n^{NI}$ be given and consider $\PP=\PP(A)$. We know that the columns of $A_0$ represent some vertices of $\PP$. The vertex
represented by $\col(A_0,j)$ will be denoted $\underline{j}$, for $j\in[n]$. Vertices of $\PP$ represented by columns of $A_0$ are
called \textbf{generators} of $\PP$. The columns of $(A^T)_0$ represent vertices of $\PP$ (see \cite{Puente_linear}). The vertex of
$\PP$ represented by $\col((A^T)_0,j)$ is  denoted $\underline{12\ldots j-1\; j+1\ldots n}$\footnote{The order of digits is unimportant.}. Vertices of $\PP$ represented by columns of $A_0$ or
$(A^T)_0$ are
called \textbf{principal}. If $\PP$ is maximal, then the number of vertices of $\PP$ is ${{2n-2}\choose {n-1}}$
(see \cite{Develin_Sturm}), so that some vertices of $\PP$ are \textbf{non--principal}. \footnote{Recall that $\underline{123}=\NN$  and  $\underline{4}=\SSSS$, i.e.,  the South Pole is a generator in $\PP$.}

For $n=4$,  let us label the \textbf{non--principal vertices}. Consider $i\neq j\in [4]$.
The shortest edge--path joining   generators $\underline{i}$  and $\underline{j}$ contains three edges and the vertex closest to
$\underline{i}$ is labeled $\underline{ij}$\footnote{The order of digits is important.}, \label{dfn:comb_labels} the vertex
closest  to
$\underline{j}$ is labeled $\underline{ji}$ (see \cite{Jimenez_Puente}). So, non--principal vertices are
$\underline{12}$, $\underline{21}$, $\underline{13}$, $\underline{31}$,\ldots,$\underline{14}$, $\underline{41}$.\footnote{If $n>4$, we
do not have a general rule to label the  non--principal vertices of $\PP$.}
\end{dfn}

\subsection{The 2--minors of a matrix}\label{subsec:2_min}

In this subsection we introduce the 2--minors  of the NI matrix $A$, some of which turn out to be edge--lengths of $\PP(A)$.
\begin{dfn}\label{dfn:2_minors}
\begin{enumerate}
\item The \textbf{difference} of matrix $\begin{pmatrix}
    a&b\\c&d
    \end{pmatrix}$ is $a+d-b-c$. \footnote{This is related to, but different from,  the tropical determinant (also called tropical permanent).}
\item The \textbf{2--minors of a matrix} $A=(a_{ij})$ are  the differences of the  order 2 submatrices  of $A$.
If  $i,j,k,l\in [n]$, $i\le j$ and $k\le l$, we write
\begin{equation}\label{eqn:2_minors_1}
a_{ij;kl}:=a_{ik}+a_{jl}-a_{il}-a_{jk}, \quad \text{and}
\end{equation}
\begin{equation}\label{eqn:2_minors_2}
 a_{ij;kl}=-a_{ji;kl}=-a_{ij;lk}=a_{ji;lk}.
 \end{equation}
 \end{enumerate}
\end{dfn}
The 2--minors satisfy some of the usual properties of classical determinants, such as  $a_{ij;kl}=0$, if $i=j$ or $k=l$.
Besides, \textbf{cocycle relations} hold
\begin{equation}\label{eqn:cocycle}
a_{ij;kl}=a_{im;kl}+a_{mj;kl}=a_{ij;km}+a_{ij;ml}.
\end{equation}
From (\ref{eqn:ij_visualization}), we get
\begin{equation}\label{eqn:entries_of_visual}
 (\leftidx{^D}A)_{ij}=-a_{in;ij}
 \end{equation} since $a_{ii}=0$, so the entries of the visualization of $A$ are certain 2--minors of $A$.

The following lemma is crucial: it is used in p.~\pageref{use:lem:2_minors}, \pageref{use:lem:2_minors_2} and \pageref{use:lem:2_minors_3}.
\begin{lem}[Invariance]\label{lem:2_minors}
The 2--minors  of a normal idempotent matrix are invariant under visualization.
\end{lem}
\begin{proof}
Apply definition  (\ref{eqn:2_minors_1})  and (\ref{eqn:2_minors_2}) to the entries (\ref{eqn:ij_visualization}) and verify cancelations.
\end{proof}

\section{From boxes to alcoved polyhedra}\label{sec:from_box_to_poly}

In English, the regular verb \textbf{to cant} means  to bevel, to form an oblique surface upon something. We will
cant some edges in  a box, making two dihedral angles of $3\pi/4$ radians (i.e., $135^\circ$). The result will be  an alcoved polyhedron.
Not every edge can be canted: only those edges not meeting the Poles can. Canting an edge depends on a parameter. The
expressions  \textbf{canted box}
and  \textbf{perturbed box} are synonyms.

\begin{dfn}[Perturbation of a VI box matrix]\label{dfn:perturbation}
Given $A=(a_{ij})\in M_4^{VI}$, we write
\begin{equation}\label{eqn:box_and_pert}
B:=B(a_{14},a_{24},a_{34}) \text{\ as \ in \ (\ref{eqn:box_matrix_1})\ and \  (\ref{eqn:box_matrix_2})\  and \ } E:=B-A.
\end{equation}
 We say that  $\PP(B)$ is the \textbf{bounding box} of $\PP(A)$,
that  $B$ is the \textbf{bounding box matrix} of $A$.
We also say that  $E$ is the  \textbf{perturbation matrix} of $A$.
\end{dfn}

The box matrix $B$ is determined by the last column of the given matrix $A$. The perturbation matrix
$E=(e_{ij})$  satisfies
\begin{equation}\label{eqn:eij}
e_{i4}=e_{4i}=e_{ii}=0 \text{\ and \ } e_{ij}=a_{i4}-a_{ij}\le 0, \text{\ all \ } i,j\in[4].
\end{equation}
The above inequalities $e_{ij}\le0$ follow
from the hypothesis that $A$ is visualized idempotent. The matrix $B$ is visualized idempotent but $E$ is not, in general. We have
\begin{equation}\label{eqn:eij_as_2_minor}
e_{ij}=-a_{i4;j4} \text{\ and\ } e_{ij;kl}=-a_{ij;kl}, \text{\ all \ } i,j,k,l\in[3].
\end{equation}

\begin{dfn}[Cant and difference tuples of $E$]\label{dfn:cant_tuple}

If $E$ is a $4\times 4$ perturbation matrix as in (\ref{eqn:eij}),
the \textbf{cant tuple} of $E$  is  $c=(c_j)\in \R_{\le0}^6$, with
$c_1:=e_{23}$, $c_2:=e_{13}$, $c_3:=e_{12}$, $c_4:=e_{32}$, $c_5:=e_{31}$ and $c_6:=e_{21}$.
The \textbf{difference tuple} is $d=(d_i)\in \R^6$, $d_i=c_{i+1}-c_{i}$, where  $c_7=c_1$.\footnote{In this section we take $n=4$,
because we do not know how to make definition \ref{dfn:cant_tuple} in more generality.}
\end{dfn}
In other words, start, say,  at entry $(2,3)$--th in $E$,
then go to the $(2,1)$--th entry, etc., and thus gather the entries of $E$ into the cant tuple
$c$, running in a \textbf{crocked closed path} (see an example in p.~\pageref{ex:example}).
\label{com:crocked}
Using $e_{i4}=e_{4i}=0$ from (\ref{eqn:eij}), we get
\begin{equation}\label{eqn:di_formulas_e}
d_I=
e_{i,i+1;i-1,4},\quad I\in\{1,3,5\},\quad d_I=
e_{i-1,4;i,i+1},\quad I\in\{2,4,6\},
\end{equation}
where  $i=I \mod 3$ and indices in $e_{ij}$  are reduced modulo 3.

\begin{dfn}[Cant and difference tuples of $A$]\label{dfn:cant_A}
For $A=B-E\in M_4^{VI}$ as in (\ref{eqn:box_and_pert}), the cant and difference tuples of $A$ are, by definition, the cant and difference tuples of $E$.\footnote{This makes sense, because a box is an unperturbed alcoved polyhedron.}
\end{dfn}
From (\ref{eqn:eij_as_2_minor}) it follows
\begin{equation}\label{eqn:di_formulas_a}
d_I=
-a_{i,i+1;i-1,4},\quad I\in\{1,3,5\},\quad d_I=
-a_{i-1,4;i,i+1},\quad I\in\{2,4,6\}.
\end{equation}
Explicitly,
\begin{equation}\label{eqn:di_values}
d_1=-a_{12;34},\ d_3=-a_{31;24},\ d_5=-a_{23;14}, \ d_2=-a_{14;23},\ d_4=-a_{34;12},\ d_6=-a_{24;31}.
\end{equation}

\begin{figure}[ht]
 \centering
  \includegraphics[width=7cm]{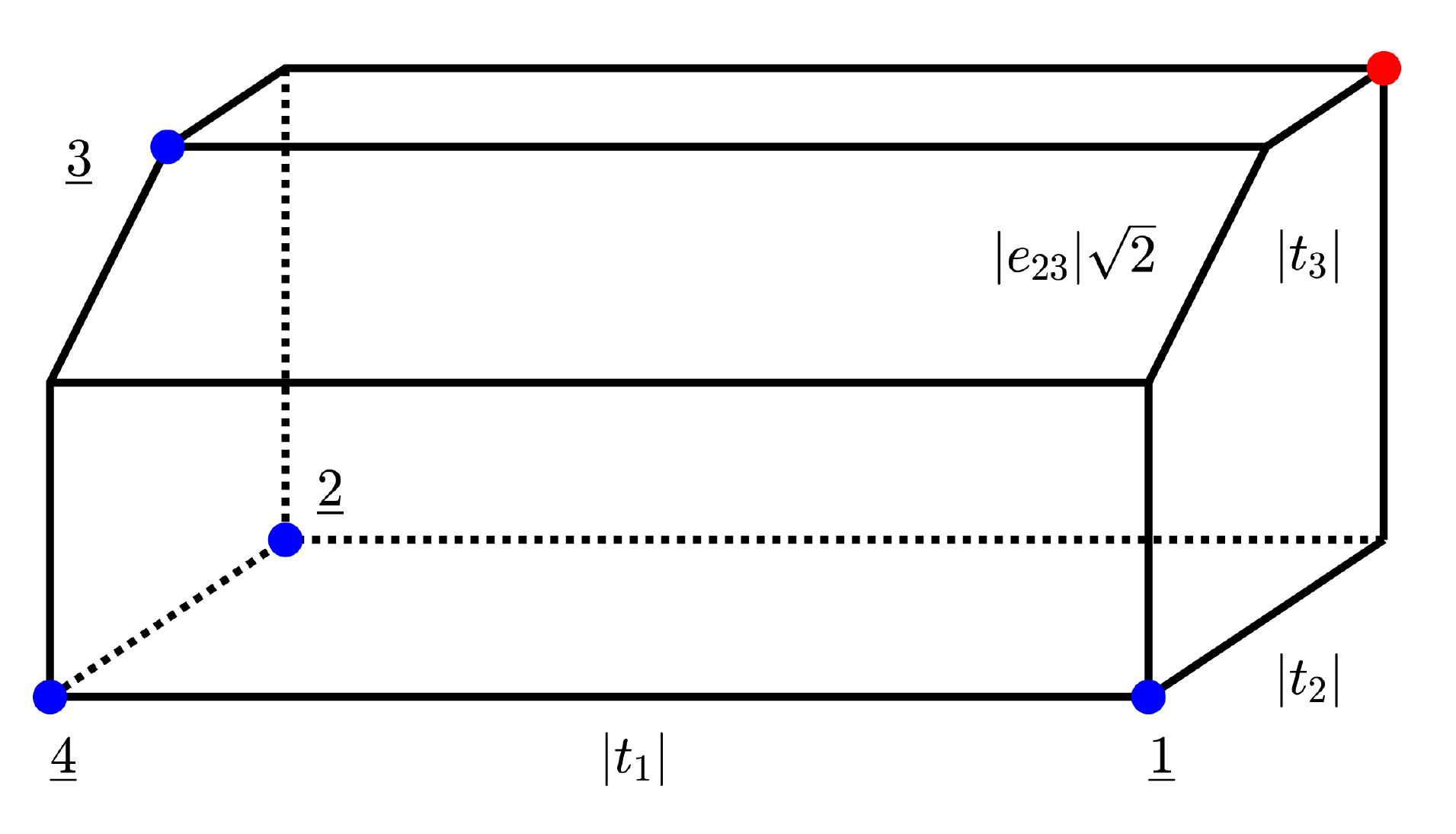}\\
  \caption{One cant  performed on the box $\BB=\PP(B(t))$  yields the alcoved polyhedron $\PP(A)$, for  matrix $A$ in example \ref{ex:one_cant}.
  The Euclidean (resp. tropical) length of the new edge is $|e_{23}|\sqrt{2}$ (resp.  $|e_{23}|$). The angles of the pentagonal facets in $\PP(A)$ are $3\pi/4$ and $\pi/2$. There is a left pentagon meeting $\NN=\underline{123}$ and a right pentagon meeting $\SSSS=\underline{4}$.}\label{fig_04_one_cant_new}
\end{figure}

\begin{ex}[One cant]\label{ex:one_cant}
From a box $\BB$, we will  remove a right prism, whose base is a right isosceles triangle, legged $|e_{23}|$
(see figure \ref{fig_04_one_cant_new}). Given the box matrix $B=B(t)$, let us \textbf{cant the top front edge of the box $\BB=\PP(B)$,
using  the parameter $e_{23}<0$} such that $|e_{23}|\le \min\{|t_2|, |t_3|\}$. We have
$A=\left(\begin{array}{rrrr}0&t_1&t_1&t_1\\t_2&0&t_2-e_{23}&t_2\\t_3&t_3&0&t_3\\0&0&0&0\end{array}\right)$. Let
$e_{ij}=0$, otherwise.
The resulting polyhedron $\PP(A)$ has one more facet than $\BB$ and  two 4--gon facets  in the box $\BB$ become 5--gon facets  in $\PP(A)$.
\end{ex}

\begin{dfn}[Cantable edge]\label{dfn:cantable_edge}
An edge $\ell$ in a box $\BB\subset \R^3$ is \textbf{cantable} if no
end point of $\ell$ is a Pole (see figure \ref{fig_03_three_cubes} right). A box  has  six cantable edges.
\end{dfn}
The cantable edges in a box form a \textbf{Petrie polygon}, i.e., a polygon dividing the box into two congruent parts. It is also
a \textbf{3--equilateral spacial polygon} and  an \textbf{antiprismatic polygon}, see \cite{Prasolov_Tikh,Schulte}.
Assuming the box matrix is (\ref{eqn:box_matrix_2}), the equations of the cantable edges are:   $\{x_3=0, x_2=t_2\}$ for
$\ell_1$,\; $\{x_3=0, x_1=t_1\}$ for $\ell_2$,
\; $\{x_2=0, x_1=t_1\}$ for $\ell_3$, \; $\{x_2=0, x_3=t_3\}$ for $\ell_4$, \;$\{x_1=0, x_3=t_3\}$ for $\ell_5$ and
$\{x_1=0, x_2=t_2\}$ for $\ell_6$.

\medskip

Box, perturbation, cant and difference tuples have been defined for visualized idempotent matrices but  it is good to have analogous
notions for NI matrices.

\begin{lem}[Box and perturbation of visualization of a NI matrix]\label{lem:visualizing_2}
Let $V=(v_{ij})\in M_4^{VI}$ be the visualization   of $A=(a_{ij})\in M_4^{NI}$ and write $V=B-E$ (box
and perturbation, as in (\ref{eqn:box_and_pert})). Then
\begin{equation}\label{eqn:s_and_a}
v_{ij}=-a_{i4;ij},\ i,j\in[4] \text{\ and \ } e_{ij}=-v_{i4;j4}=-a_{i4;j4}, \ i,j\in[3].
\end{equation}
In particular, $v_{i4}=a_{i4}+a_{4i}$, $i\in[4]$,  are the fourth column entries of $V$.
\end{lem}
\begin{proof}
The first equality in (\ref{eqn:s_and_a}) is (\ref{eqn:entries_of_visual}); the second and third equalities follow from the first one,
 (\ref{eqn:eij_as_2_minor}) and the cocycle relation (\ref{eqn:cocycle}). Alternatively, one can use lemma \ref{lem:2_minors}. \label{use:lem:2_minors}
\end{proof}
The moral is that certain 2--minors of $A\in M_4^{NI}$ provide both the edge--lengths of the box of $\PP(A)$ as well as the cant parameters. The following definition is a consequence of $e_{ij}=-v_{i4;j4}=-a_{i4;j4}$  in (\ref{eqn:s_and_a}).
\label{why_2_minors}\label{moral_1}

\begin{dfn}[Perturbation, cant and difference of NI matrix]\label{dfn:perturbation_of_NI}
The perturbation matrix of $A\in M_4^{NI}$ is, by definition,  the perturbation matrix $E$ of the visualization  $V=\leftidx{^D}A=B-E$ of $A$, where $D=\diag(\row(A,n))$.
The cant and difference tuples of $A$ are  the cant and difference tuples of $E$.
\end{dfn}

\section{Symmetries of a cube}\label{sec:symmetries}
The contents of this section are well--known. Let $(u_1,u_2,u_3)$ be the canonical basis in $\R^3$. Let $\Sigma_3, \Sigma_4$ denote the
permutation groups in 3, 4 symbols.
The unit  cube centered at the origin $\QQ\subset \R^3$ has 24 \textbf{orientation--preserving  symmetries} (called \textbf{rotations},
for short). Indeed, each
rotation of $\QQ$  permutes the  diagonals of $\QQ$ (see in figure \ref{fig_03_three_cubes}, center, the diagonals of a cube,
labeled 1, 2, 3, 4 not underlined). The \textbf{full group of symmetries} of $\QQ$ is the
direct product $\GG_4=\Z_2\times \Sigma_4$. So, in addition to rotations, $\GG_4$ contains all the
\textbf{orientation--reversing  symmetries}, such as the antipodal map and orthogonal reflections on  planes.

The element $(\id,\id)\in \GG_4$ represents the identity map on the cube,  while  the element $(-\id,\id)$
represents the \textbf{antipodal map}
$-\id:\QQ \rightarrow \QQ$.
For $\tau\in \Sigma_4$, the element
$(\id,\tau)$ represents
the \textbf{orientation--preserving symmetry} permuting the diagonals of
$\QQ$  according to $\tau$. On the other hand, the element $(-\id,\tau)$ represents the \textbf{orientation--reversing symmetry} permuting
the diagonals of $\QQ$  according to $\tau$. In this section, we write  $+\tau$ instead of $(\id,\tau)$, and  $-\tau$
instead of $(-\id,\tau)$, for convenience.

\begin{figure}[ht]
 \centering
  \includegraphics[width=12cm]{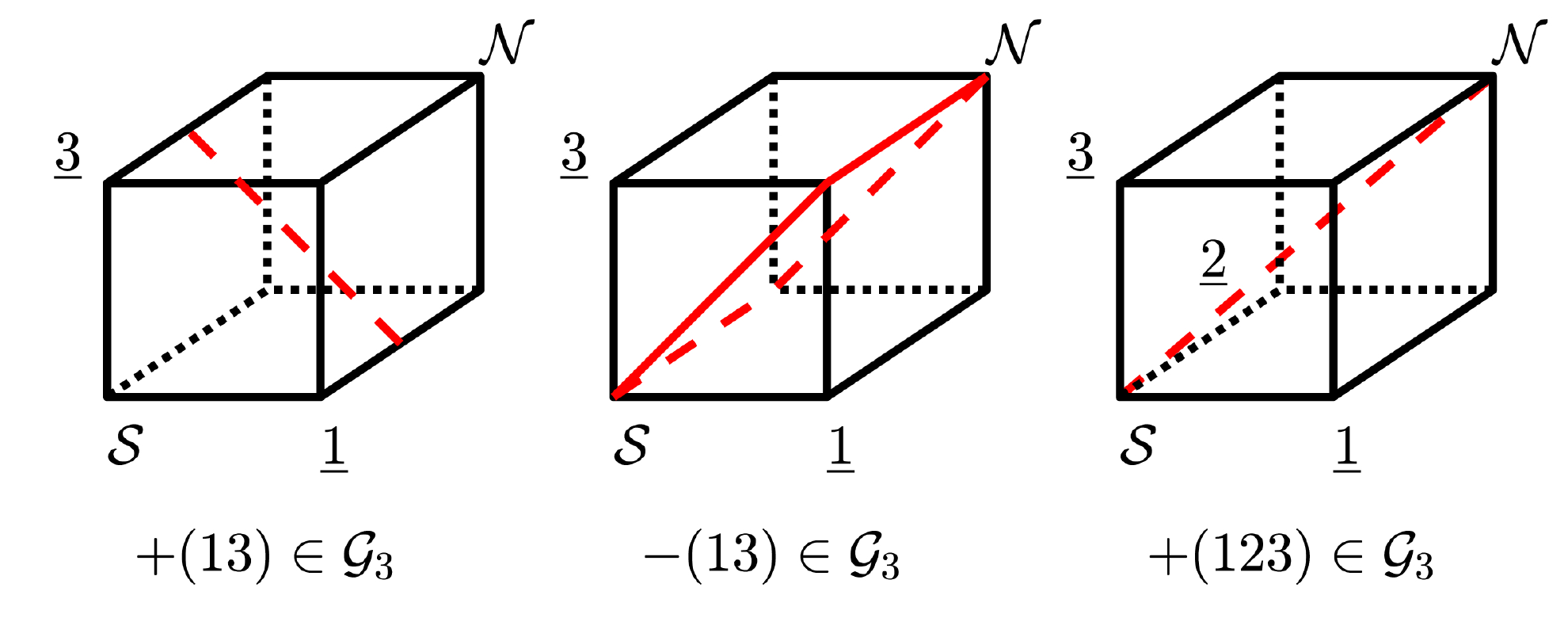}\\
  \caption{A Polar Exchange  rotation is $+(13)$ (left). The orthogonal reflection on plane $x_1=x_3$ is $-(13)$ (center). The $120^\circ$ rotation around the Polar Axis (in the positive sense) is $+(123)$  (right). The invariant sets are marked in red.}\label{fig_05_three_cubes_2}
\end{figure}

Recall that $(u_1,u_2,u_3)$ is the canonical basis in $\R^3$.
For our purposes, we only need cube symmetries preserving the Polar Axis (which is the diagonal labeled 4)
so, we  work with the subgroup \label{dfn:group}
$\GG_3:=\Z_2\times \Sigma_3$.\footnote{$\GG_3$ is isomorphic to a  dihedral group of order 12.}
Since $(12),(13)$ generate the group $\Sigma_3$, then $+(12),+(13),-\id$ generate $\GG_3$.
Here is a description of these group elements,
as  cube symmetries:
\begin{enumerate}
\item  for $i\neq j\in [3]$,  the element $+(ij)$ permutes the diagonals $i$ and $j$ of $\QQ$, preserving orientation. It is easily seen that it corresponds to  the  \textbf{$180^\circ$ rotation around the  axis line spanned by the vector $u_i-u_j$} (see red dashed line in figure \ref{fig_05_three_cubes_2}, left).
It transforms the North Pole into the South Pole. It transforms the facet $x_i=1$ into  the facet $x_j=-1$ and the
facet $x_k=1$ into  the facet $x_k=-1$, for $i\neq k\neq j$,
(see figure \ref{fig_05_three_cubes_2}, left). Such a map is called a \textbf{Polar Exchange}. There are three Polar Exchanges: $+(12), +(13)$ and $+(23)$. \label{dfn:polar_exch}

\item  for $i\neq j\in [3]$, the element $-(ij)$ is the  orthogonal reflection
on the plane $x_i=x_j$. The effect  on $\QQ$ is to produce a \textbf{chiral}  \label{item:chiral} copy of $\QQ$. \footnote{A chiral copy of a rubber glove $G$ is obtained by turning $G$ inside out. Alternatively, we cut $G$ along a meridian, we fold the two pieces inside out and then we glue them again. A chiral copy of the cube $\QQ$ is obtained similarly.}
\end{enumerate}

\begin{rem}\label{rem:cantable_go_to_cantable}
If the edge $\ell$ is cantable in  $\QQ$, then  the edge  $g(\ell)$ is cantable in $\QQ=g(\QQ)$,   for all  $g\in \GG_3$.
\footnote{This is not true for all $g\in \GG_4$; for instance, take $g=+(1234)$.}
\end{rem}

\subsection{Matrix transformations vs. cube symmetries}\label{subsec:transf}

The following lemma is obvious.
\begin{lem}\label{lem:transpose}
If $A\in M_4^{NI}$, then $A^{T}\in M_4^{NI}$. Furthermore,  $\underline{i}_A\mapsto \underline{jkl}_{A^T}$  is a bijection between the
sets of principal  vertices of $\PP(A)$ and
$\PP(A^{T})$, with  $\{i,j,k,l\}=[4]$. \qed
\end{lem}

\begin{cor}[Antipodal map; Corollary 4 in \cite{Jimenez_Puente}]\label{cor:symmetry}
If $A\in M_4^{NI}$, then $A^{T}\in M_4^{NI}$ and
\begin{equation}\label{eqn:sym}
\PP(A^T)=-(\PP(A)).
\end{equation} In particular,  the matrix $A$ is symmetric if and only if  the alcoved polytope $\PP(A)$ is  central--symmetric
with respect to the origin.\qed
\end{cor}

\begin{ex}\label{ex:unit_cube}
The matrices of the centered unit cube $\QQ$ (of edge--length 2)  are
$Q_0=\begin{pmatrix}
1&-1&-1&-1\\
-1&1&-1&-1\\
-1&-1&1&-1\\
0&0&0&0
\end{pmatrix}
$, $Q=\begin{pmatrix}
0&-2&-2&-1\\
-2&0&-2&-1\\
-2&-2&0&-1\\
-1&-1&-1&0
\end{pmatrix}\in M_4^{NI}$ and the visualization of $Q$ is $\leftidx{^D}Q=\begin{pmatrix}
0&-2&-2&-2\\
-2&0&-2&-2\\
-2&-2&0&-2\\
0&0&0&0
\end{pmatrix}.$
\end{ex}

The  \textbf{group} $\GG_3=\Z_2\times \Sigma_3$ (defined in p.~\pageref{dfn:group}) \textbf{acts}\footnote{The action of a group $G$ on a set $S$ is a map $G\times S\rightarrow S$, $(g,s)\mapsto g\cdot s$, such that (a) $\id\cdot s=s$, all $s\in S$ and  (b) $g_1\cdot(g_2\cdot s)=(g_1 g_2)\cdot s$, all $g_1, g_2\in G$, all $s\in S$.} on $M_4^{NI}$ as follows:  \label{dfn:group_action_1} for $A\in M_4^{NI}$ and $i\neq j\in [3]$,
 \begin{enumerate}
 \item $-\id\cdot A:=A^T$, \label{item:alpha}
 \item $-(ij) \cdot A:=\leftidx{^{(ij)}}A$, (permutation of two generators) \label{item:beta}
 \item $+(ij) \cdot A:=\leftidx{^{(ij)}}A^T$, (Polar Exchange). \label{item:gamma}
 \end{enumerate}
 Indeed, item \ref{item:alpha} is a restatement of (\ref{eqn:sym}). To explain items \ref{item:beta} and \ref{item:gamma}, recall that
  the element  $+(ij)\in\GG_3$   permutes the diagonals $i$  and $j$ of the cube $\QQ$;
 however it does not permute the generators $\underline{i}$ and $\underline{j}$. It is the group element $-(ij)$ who does this job.

\begin{rem}[Chirality]\label{rem:chiral}
Let $A\in M_4^{NI}$.
A  \textbf{chiral} copy of $\PP(A)$ is given by $\PP (\leftidx{^{(ij)}}{A})$, for $i\neq j\in[3]$. Two chiral copies of $\PP(A)$
are related by a power of $+(123)\in \GG_3$. Here   $+(123)$ is  the  $120^\circ$
rotation around the Polar Axis (in the positive sense); see figure \ref{fig_05_three_cubes_2} (right).
\end{rem}

\section{Maximality}\label{sec:maximality}
If $\PP$ is a maximal alcoved polyhedron in $\R^3$, then in \cite{Jimenez_Puente} (see also \cite{Joswig_Kulas,Tran}) it is  proved that
\begin{enumerate}
\item  the $f$--vector of $\PP$ is $(f_0,f_1,f_2)=(20,30,12)$; in particular $\PP$ is a \textbf{dodecahedron} and it is \textbf{simple} (i.e., trivalent in vertices), because $20\times3= 30\times2$,
\item  the facets of $\PP$  are alcoved 4--gons, 5--gons and 6--gons, with \textbf{facet  vector} (also called $p$--vector) $p=(p_4,p_5,p_6)=(2,8,2)$, $(3,6,3)$ or $(4,4,4)$, where  $f_2=p_4+p_5+p_6$ and  $p_j$ is the number of $j$--gons. \label{dfn:p_vector}
\end{enumerate}
Thus, every maximal alcoved polyhedron is a dodecahedron, but the converse is not true, in general.

\section{Tropical distance}\label{sec:dist}
We will use \textbf{tropical distance}\footnote{$\dd(p,q)$ is the maximum of the Chebyshev distance $\dd_{Ch}(p,q):=\max_{i,j\in[n]}\{|p_i-q_i|\}$ and a tropical version of the Hilbert  projective distance $\dd (p,q)_{pr.H}:=\max_{i,j\in[n]}\{|p_i-q_i-p_j+q_j|\}$.} to measure edge--lengths:
\begin{equation}\label{dfn:trop_dist}
\dd (p,q):=\max_{i,j\in[n]}\{|p_i-q_i|, |p_i-q_i-p_j+q_j|\}, \quad p,q\in\R^n.
\end{equation}
For example, we have $\dd ((1,1),(0,0))=1$ (not $\sqrt{2}$!), $\dd ((1,1,1),(0,0,0))=1$ (not $\sqrt{3}$!),  and
$\dd ((-5,-2),(-2,-5))=6$. See caption in figure \ref{fig_04_one_cant_new}.

Tropical distance between integral points is just integer (i.e., lattice) distance.

This distance (and a related seminorm and  norm) have been used  since 1979 in \cite{Akian_Gau_Nit_Sin,Cohen_al,Cuninghame,Cuninghame_New,Cuninghame_But,Develin_Sturm,Joswig_Sturm_Yu,Krivulin_Ser,Puente_kleene,Puente_line}.
In the literature, they are  called by various names, such as \emph{tropical Hilbert    projective distance} or
\emph{Chebyshev distance} and \emph{range seminorm}. In \cite{Deza} we find the \emph{Hilbert projective metric} on the space of rays of
a convex cone.
\begin{rem}\label{rem:2_minors}
If $i\le j$, then $p_i-q_i+p_j-q_j$ is the 2--minor of $\begin{pmatrix}
p_i&q_i\\
p_j&q_j
\end{pmatrix}$. In particular, if $p_n=q_n=0$, then $\dd(p,q)$ is the maximum of the absolute values of the 2--minors of the matrix whose columns are the coordinates of $p$ and $q$.
\end{rem}

\section{Polar Casks and Equatorial Belt}\label{sec:casks_and_belt}

\begin{dfn}\label{dfn:Cask}
In the alcoved polyhedron $\PP\subset\R^3$, the \textbf{North Cask}
(resp. \textbf{South Cask})  is the union of the  facets of $\PP$ containing the North Pole
(resp. the South Pole).
\end{dfn}

\begin{figure}[ht]
\centering
\includegraphics[keepaspectratio,width=7cm]{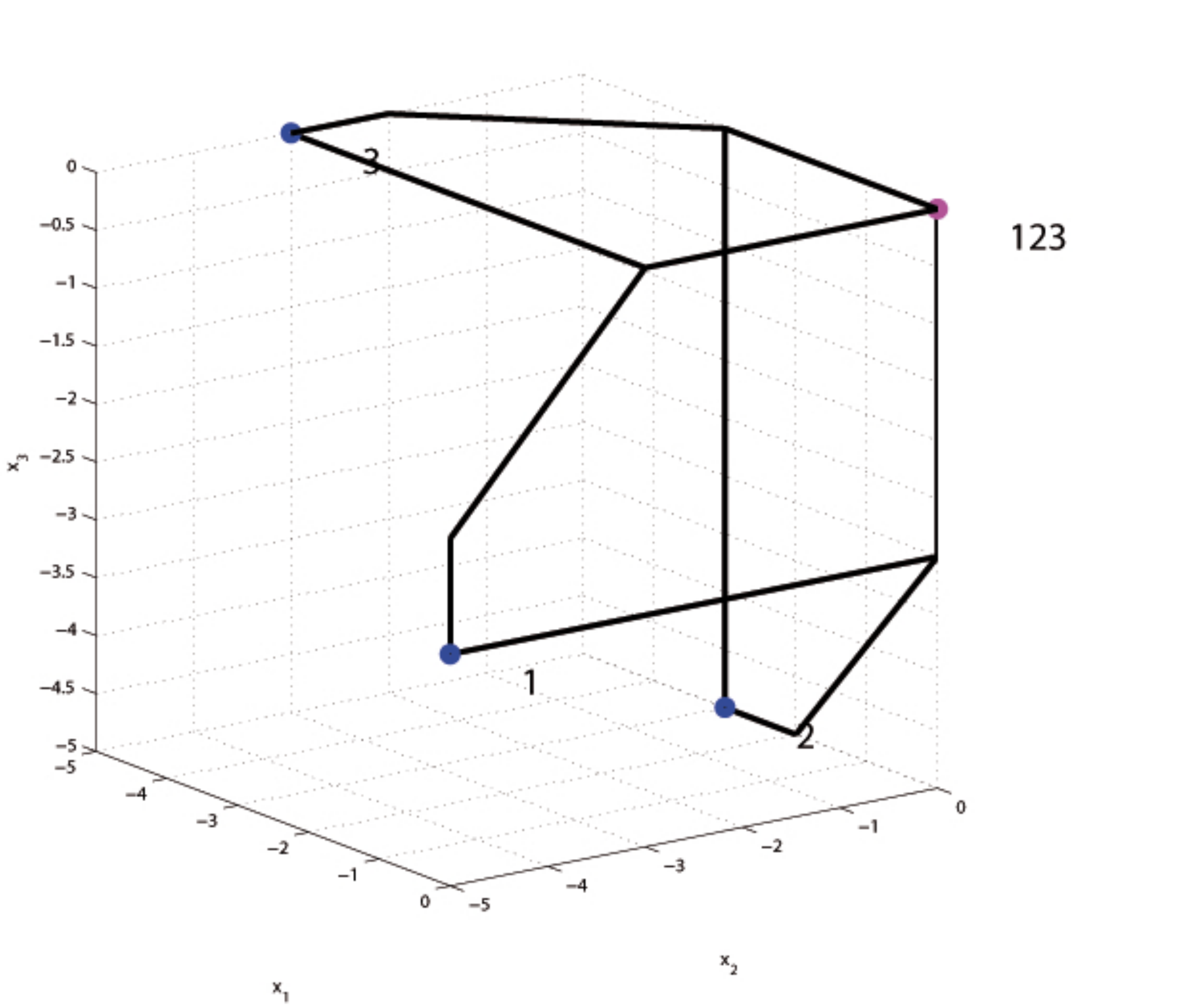}\\
\caption{North Cask of type $(5.5.5)$ left. This means that three left 5--gons meet at the North Pole. Note that the dihedral angles in the North (resp. South) Cask are $90^\circ$ each. Cf. figure \ref{fig_07_types_N}.}
\label{fig_06_N_Cask}
\end{figure}

\begin{dfn}[Type of a Polar Cask]\label{dfn:Cask_type}
In a maximal alcoved  dodecahedron $\DD\subset \R^3$, consider a Pole $\VV\in\{\NN,\SSSS\}$. Let $p,q,r\in\{4,5,6\}$ with $p+q+r=15$.
In $\VV$, if the facet in $x_1=cnst$ is a $p$--gon, the facet in $x_2=cnst$ is a $q$--gon and the facet in $x_3=cnst$ is a $r$--gon, and these facets meet at $\VV$, then
the \textbf{type of the $\VV$--Cask} is  $(p.q.r)$.
    If such a pentagon is right (resp. left) then
    the $\VV$--Cask is  \textbf{right} (resp. \textbf{left}).
    \end{dfn}

\subsection{North Cask}\label{subsec:north_cask}
For $i\in[3]$ we consider $i-1$ and $i+1$  reduced to elements in
$[3]$ modulo 3. Notations $\rho_i, \lambda_i,\delta_i,\epsilon_i$ are introduced in
the proof of the next proposition. The \textbf{Greek letter} $\lambda$ reminds us of \textbf{left}, and $\rho$ reminds us of \textbf{right}.  These letters
will represent  certain non--negative edge--lengths of some facet $x_i=cnst$, with $i\in[3]$ of the North Cask, where the subscripts of
the letters coincide with the subscripts of the variable facet equation.  Every edge--length $\rho_i$ (resp. $\lambda_i$) arises
from  canting a certain cantable edge in the bounding box. Every edge--length $\delta_i$ (resp. $\epsilon_i$) is the \textbf{complement} of a corresponding $\rho_i$ (resp. $\lambda_i$) in the sense that (\ref{eqn:delta_eps}) holds.

\begin{prop}\label{prop:type_North_Cask}
For a maximal alcoved dodecahedron $\DD=\PP(A)$ with $A\in M_4^{VI}$, the following hold:
\begin{enumerate}
\item the \textbf{type  of the North Cask} is  determined by the \textbf{signs of three 2--minors}  of $A$:
$$a_{i,4; i-1,i+1}, \quad i\in [3],$$ \label{item:un}
\item the type of the North Cask is either  $(\tau(4).\tau(5).\tau(6))$, for all  permutations $\tau$ in $\Sigma_3$, or $(5.5.5)$ right, or  $(5.5.5)$ left,\label{item:deu}
\item  the  $f$--vector of the North Cask is $(f_0,f_1,f_2)=(10,12,3)$.\label{item:cua}
\end{enumerate}
\end{prop}

\begin{figure}[H]
 \centering
  \includegraphics[width=12cm]{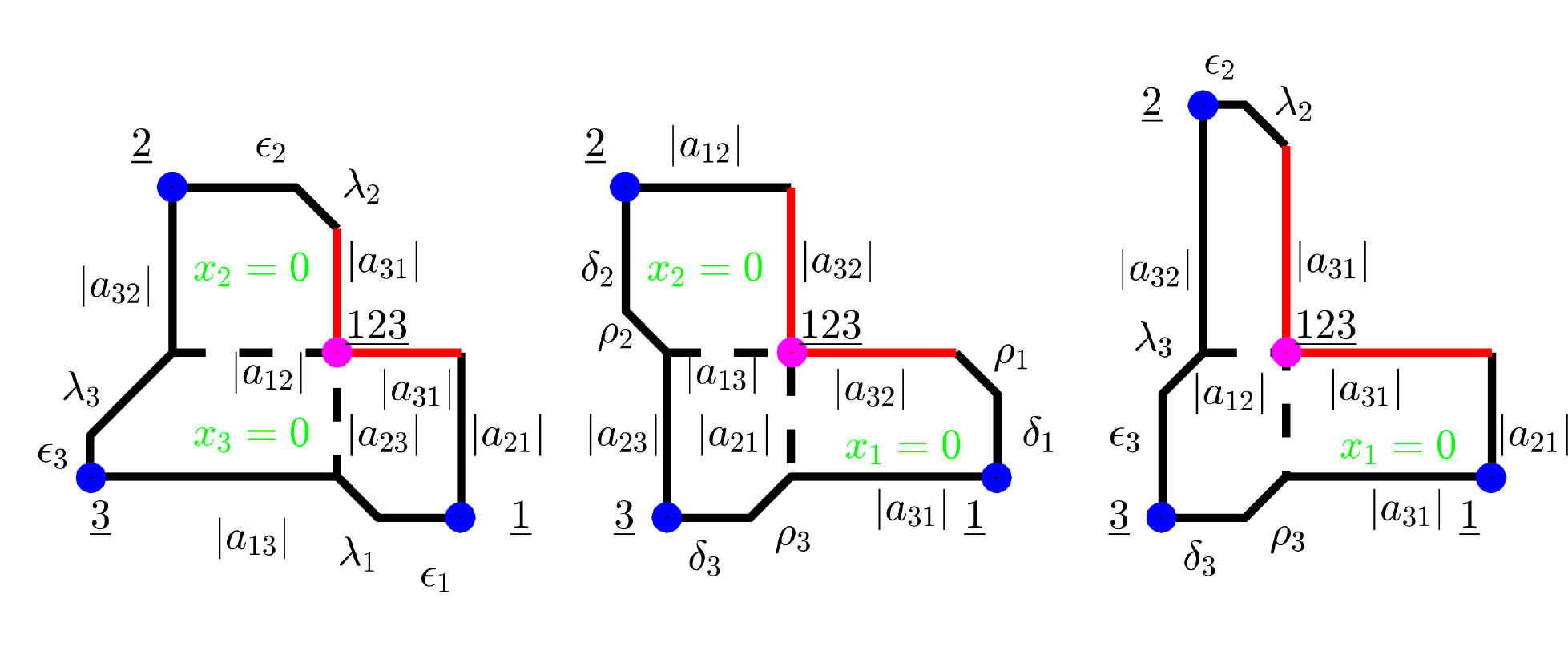}\\
  \caption{North Cask is $(5.5.5)$ left (left), $(5.5.5)$ right (center), $(4.5.6)$ (right). Dashed segments are meant to be folded down. Red segments are meant to be glued together. If the figure on the left is folded, we get the North Cask shown in  figure \ref{fig_06_N_Cask}. If, instead of folding down, we fold up, we will  get  chiral copies. Details on the center figure are explained in remark \ref{rem:details}.}
  \label{fig_07_types_N}
\end{figure}

\begin{proof}
The matrix $A$ is visualized idempotent, whence $a_{4i}=0$,  $i\in[4]$, i.e., the North Pole $\NN$ is the origin.
From (\ref{eqn:2_minors_1}) and (\ref{eqn:2_minors_2}), we get
\begin{equation}\label{eqn:3_2_minors}
a_{i,4; i-1,i+1}=a_{i,i-1}-a_{i,i+1}, \quad i\in [3].
\end{equation}
We cannot have $a_{i,4; i-1,i+1}=0$, by maximality of $\DD$ with respect to $f$--vector (i.e., the matrix $A$ is generic
inside $M_4^{VI}$).

For fixed $i\in[3]$, we \textbf{define}
\begin{enumerate}
\item if $a_{i,4; i-1,i+1}>0$, \label{case:one}
\begin{enumerate}
\item $\rho_{i+1}:=a_{i,4; i-1,i+1}$ and $\lambda_{i-1}:=-\infty$,\label{item:a}
\item $\delta_{i+1}:=a_{i,i-1;i+1,i-1}$ and $\epsilon_{i-1}:=-\infty$,\label{item:b}
\end{enumerate}
\item if $a_{i,4; i-1,i+1}<0$, \label{case:two}
\begin{enumerate}
\item $\rho_{i+1}:=-\infty$ and $\lambda_{i-1}:=-a_{i,4; i-1,i+1}$, \label{item:c}
\item $\delta_{i+1}:=-\infty$ and $\epsilon_{i-1}:=a_{i,i+1;i-1,i+1}$. \label{item:d}
\end{enumerate}
\end{enumerate}
Clearly,
\begin{equation} \label{eqn:only_3_pos}
\text{\ there\ are\ exactly\ three\ positive\ numbers\ in\ the\ set\ } \{\rho_1,\rho_2,\rho_3,\lambda_1,\lambda_2,\lambda_3\}.
\end{equation}

By the determinantal properties of 2--minors  and the cocycle relations (\ref{eqn:cocycle}), we have
\begin{equation}\label{eqn:delta_eps}
\delta_{i+1}+\rho_{i+1}=\begin{cases}
-a_{i-1,i+1}, \text{\ if\ case \ }  \ref{case:one} \text{\ above,}\\
-\infty, \text{\ otherwise,}
\end{cases} \epsilon_{i-1}+\lambda_{i-1}=\begin{cases}
-\infty, \text{\ if\ case\ }  \ref{case:one} \text{\ above,}\\
-a_{i+1,i-1}, \text{\ otherwise,}
\end{cases}
\end{equation}
and we say that  $\delta_{i+1}, \rho_{i+1}$ (resp. $\epsilon_{i-1}, \lambda_{i-1}$) are \textbf{complementary}.
Also, operating indices modulo 3, we have
\begin{equation}\label{eqn:new_lambdas}
\lambda_{i+1}=\begin{cases}
-a_{i-1,4;i+1,i}=a_{i-1,i}-a_{i-1,i+1}\\
-\infty,
\end{cases}
\epsilon_{i+1}=\begin{cases}
a_{i-1,i;i+1,i}\\
-\infty.
\end{cases}
\end{equation}
\begin{equation}\label{eqn:new_eps_lamb}
\epsilon_{i+1}+\lambda_{i+1}=\begin{cases}
-a_{i,i+1}\\
-\infty.
\end{cases}
\end{equation}

Two cases arise:
\begin{itemize}
\item \textbf{Case 1: if there exists $i\in[3]$ such that  $\rho_{i+1}>0$ and  $\lambda_{i+1}>0$}.  Then  $\delta_{i+1}$
and $\epsilon_{i+1}$  are positive (this follows from  $A$ being normal idempotent) and $\rho_{i+1}, \lambda_{i+1},\delta_{i+1}$,
 $\epsilon_{i+1}$, $|a_{i-1,i}|$ and $|a_{i,i-1}|$ are the edge--lengths
 of facet $x_{i+1}=0$, which makes a 6--gon in  the North Cask of $\DD$.\footnote{We explain the instance $i=2$, the other ones
 being similar (see figure  \ref{fig_07_types_N}, right). The coordinates of the generator $\underline{3}$ are $\begin{pmatrix}
 a_{13}\\a_{23}\\0
 \end{pmatrix}$ and $\delta_3+\rho_3=-a_{13}$ follows from (\ref{eqn:delta_eps}), $\epsilon_3+\lambda_3=-a_{23}$ follows
 from (\ref{eqn:new_eps_lamb}). Besides, $\rho_3+|a_{21}|=|a_{23}|$  follows from (\ref{eqn:3_2_minors}) and case
 \ref{item:a}. Similarly, we have $\lambda_3+|a_{12}|=|a_{13}|$.} Geometrically, the effect on the bounding box of $\DD$ is canting the left back edge $\ell_3$ by cant parameter $\lambda_3$ and the right front edge $\ell_6$ by cant parameter $\rho_3$ (recall caption in figure \ref{fig_03_three_cubes}).

 Further, $\rho_{i-1}<0$ and  $\lambda_{i-1}<0$ must hold,
 by items (\ref{item:a}) and (\ref{item:c}), in which case the facet $x_{i-1}=0$ makes a 4--gon in the North Cask of $\DD$, with
 edge--lengths $|a_{i,i-1}|$ and $|a_{i+1,i-1}|$. In addition, we have $\rho_i>0$ or  $\lambda_i>0$ but, because of (\ref{eqn:only_3_pos}), not both. Thus, the facet $x_{i}=0$ makes a 5--gon in the North Cask of $\DD$.
   \label{case:1}
  \item  \textbf{Case 2: if no  $i\in[3]$ exists such that  $\rho_{i+1}>0$ and  $\lambda_{i+1}>0$}. Then either $\rho_i>0$, all $i\in[3]$ or  $\lambda_i>0$, all $i\in[3]$, because (\ref{eqn:only_3_pos}) holds. These two subcases are chiral to each other.
  A similar argument to the one given in the former case shows that  the facet in $x_i=0$ is a 5--gon in
  $\DD$.   This 5--gon is left if and only if  $\lambda_i>0$, in which case $\epsilon_i>0$ and  $\lambda_i, \epsilon_i, \lambda_i+\epsilon_i, |a_{i+1,i}|$ and  $|a_{i+1,i-1}|$ are the edge--lengths.  This 5--gon is right if and only if  $\rho_i>0$ (see details in figure \ref{fig_07_types_N}, left and center).  \label{case:2}
\end{itemize}
In either case, the North Cask of $\DD$ consists of the three facets $x_i=0$, $i\in[3]$ as described above, all meeting at $\NN$, the origin.

According  to
definition \ref{dfn:Cask_type}, the possible North Cask types of $\DD$ are
\begin{itemize}
\item $(4.5.6)$ as well as all permutations of it: $(\tau(4).\tau(5).\tau(6))$, $\tau\in\Sigma_3$,
\item $(5.5.5)$  right,  if $\rho_i>0$, all $i\in [3]$, or  $(5.5.5)$  left, if $\lambda_i>0$, all $i\in [3]$, (see figure \ref{fig_06_N_Cask}).
\end{itemize}
This proves items \ref{item:un} and \ref{item:deu} in the proposition.
Vertices,  edges and facets   in the North Cask of $\DD$ are easily
counted  giving an $f$--vector equal to $(10,12,3)$\footnote{Note $1=10-12+3$, i.e., a North Cask has the  Euler characteristic of a closed disc.} (see figures  \ref{fig_06_N_Cask} and \ref{fig_07_types_N}).  This proves item \ref{item:cua}.
\end{proof}

\begin{rem}\label{rem:odds_and_ends}
\begin{enumerate}
\item In the previous proof, we have encountered two cases: \label{item:cases}
\begin{itemize}
\item $\{p,q,r\}=\{4,5,6\}$,
\item $6\notin\{p,q,r\}$  and $p=q=r=5$,
\end{itemize}
the second case splitting into two mutually chiral ones.
\item Any pentagon belonging to  a Polar Cask of type $(5.5.5)$ left (resp. right) is a \textbf{left}
(resp. \textbf{right}) \textbf{pentagon}.
 \item There are  exactly three positive numbers in the
 set $\{\epsilon_1,\epsilon_2,\epsilon_3,\delta_1,\delta_2,\delta_3\}$.
 \item In $\PP(A)$, the direction of an edge of tropical length $\rho_i$ or $\lambda_i$ (if any) is  $u_{i-1}+u_{i+1}$, all $i\in[3]$ (easy to check). \label{item:last}
\end{enumerate}
\end{rem}

\begin{rem}\label{rem:details}
Here we give some details about figure \ref{fig_07_types_N}, center. Recall that $a_{ij}\le0$. In the figure, we can see that
\begin{equation}
|a_{12}|-|a_{13}|>0 \text{\ and \ } |a_{23}|-|a_{21}|>0 \text{\ and \ } |a_{31}|-|a_{32}|>0,
\end{equation}
or equivalently
\begin{equation}
-a_{12}+a_{13}>0 \text{\ and \ } -a_{23}+a_{21}>0 \text{\ and \ } -a_{31}+a_{32}>0,
\end{equation}
which tell us that the quantities in  (\ref{eqn:3_2_minors}) are positive, all $i\in[3]$. In other words,  $\rho_i>0$, all
$i\in[3]$, whence $\lambda_i=-\infty$, because (\ref{eqn:only_3_pos}) holds.
By the proof of proposition \ref{prop:type_North_Cask}, this is the second case in Item \ref{item:cases} in Remark \ref{rem:odds_and_ends}, whence the type of the North Cask is (5.5.5) right (recall $\rho$ stands for right).

Geometrically, the effect on the bounding box $\BB$ of $\DD$ is canting the top left  edge $\ell_2$ of $\BB$ by cant parameter $\rho_2$, canting the  back
bottom edge $\ell_4$ by cant parameter $\rho_1$ and canting the right front edge $\ell_6$ by cant parameter $\rho_3$
(cf. figure \ref{fig_03_three_cubes}).
\end{rem}

Table \ref{table:types_N} summarizes the North Cask types of a maximal dodecahedron $\DD$, up to permutations in $\Sigma_3$, for a
matrix $A\in M_4^{VI}$.
\begin{table}[ht]
\begin{center}
\begin{tabular}{| >{$}l<{$} | >{$}l<{$} | >{$}l<{$} |}
\hline
\lambda_2>0,\quad \lambda_3>0,\quad \rho_3>0&  \NN \text{\ is\ } (4.5.6)&\lambda_2=-a_{34;21},\quad \lambda_3=-a_{14;32},\quad \rho_3=a_{24;13}\\
\hline
\rho_1>0,\quad \rho_2>0,\quad \rho_3>0& \NN \text{\ is\ } (5.5.5) \text{\ right}&\rho_1=a_{34;21},\quad \rho_2=a_{14;32},\quad \rho_3=a_{24;13}\\
\hline
\lambda_1>0,\quad \lambda_2>0,\quad \lambda_3>0& \NN \text{\ is\ } (5.5.5) \text{\ left}&\lambda_1=-a_{24;13},\quad \lambda_2=-a_{34;21},\quad \lambda_3=-a_{14;32}\\
\hline
\end{tabular}
\end{center}
\caption{Types of North Cask of a maximal alcoved dodecahedron. The first line shows only the identity permutation in $\Sigma_3$. Cf. figure \ref{fig_07_types_N}.}
\label{table:types_N}
\end{table}

Table \ref{table:edge_lengths} shows the edge--lengths of the North Cask in figure \ref{fig_07_types_N} right, where $A\in M_4^{VI}$: for $i=2$, we have $\lambda_3=-a_{14;32}=-a_{13}+a_{12}$, $\rho_3=a_{24;13}=-a_{23}+a_{21}$ and for $i=1$, we have $\lambda_2=-a_{34;21}=-a_{32}+a_{31}$ by (\ref{eqn:new_lambdas}).

\begin{table}[ht]
\begin{center}
\begin{tabular} {| c | c | c |}
\hline
Facet& p--gon & Edge--lenghts\\
\hline
$x_1=0$&4--gon&$|a_{21}|$, $|a_{31}|$\\
\hline
$x_2=0$&5--gon&$ \lambda_2,\epsilon_2, |a_{32}|, |a_{12}|, |a_{31}|$\\
\hline
$x_3=0$&6--gon&$|a_{21}|,|a_{12}|,\lambda_3,\epsilon_3,\delta_3,\rho_3$\\
\hline
\end{tabular}
\end{center}
\caption{Edge--lengths of the North Cask  of type $(4.5.6)$ of a maximal alcoved dodecahedron $\PP(A)$. The chirality word is   \textbf{left} (see figures \ref{fig_02_pentagon_left_and_right} and \ref{fig_07_types_N} right).}
\label{table:edge_lengths}
\end{table}

The previous proposition and corollary apply to visualized idempotent matrices, and
the corollary below applies to  NI matrices, which are  more general. The corollary follows from the invariance of 2--minors proved
in lemma \ref{lem:2_minors}\label{use:lem:2_minors_2}. The moral is that  \emph{the North Cask type can be read off from a NI matrix;
we do not need to visualize it}.\label{moral_2}

\begin{cor}\label{cor:type_North_Cask_NI}
For a maximal alcoved dodecahedron $\DD=\PP(A)$ with $A\in M_4^{NI}$,  the type  of the North Cask is  determined by the \textbf{signs of three 2--minors} of $A$:
\begin{equation}\label{eqn:2_minors_north}
a_{i,4; i-1,i+1}, \quad i\in [3].\qed
\end{equation}
\end{cor}

We can  tell the type of the North Cask from the difference tuple, as the following corollary shows.
\begin{cor}\label{cor:type_North_Cask_VI}
For a maximal alcoved dodecahedron $\DD=\PP(A)$ with $A\in M_4^{VI}$,  the type  of the North Cask is  determined by the
\textbf{signs of the entries of the difference tuple of $A$ with even index}.
\end{cor}
\begin{proof} Rewrite (\ref{eqn:di_values}) as
\begin{equation}\label{eqn:di_par}
d_{I+1}=a_{i,4; i-1,i+1},  \text{\ with\ } I\in\{1,3,5\},\ i=I\mod 3 \text{\ and\  indices\  reduced\  modulo\  3\   in \ } a_{ij},
\end{equation}
noticing that the minus sign in (\ref{eqn:di_values}) has vanished because we have permuted two column
indices.  Now apply item \ref{item:un} in
proposition \ref{prop:type_North_Cask}.  \label{item:3_2_minors2}
\end{proof}
Explicitly, we have
\begin{equation}\label{eqn:di_par_2}
d_2=a_{14;32},\ d_4=a_{34;21},\ d_6=a_{24;13}.
\end{equation}

A simple count gives the following result.
\begin{cor}\label{cor:number_types_N}
The number of North Cask types of  maximal alcoved dodecahedra is $8=2+3!$ \qed
\end{cor}

The next table follows from the proof of proposition \ref{prop:type_North_Cask}, recalling (\ref{eqn:only_3_pos}).
\begin{table}[ht]
\begin{center}
\begin{tabular}{|>{$}c<{$}|>{$}c<{$}|>{$}c<{$}|}
\hline
\sign(\rho_1,\rho_2,\rho_3,\lambda_1,\lambda_2,\lambda_3)&\sign(d_2,d_4,d_6)&\text{North\  Cask\ type}\\
\hline
(+++---)&(+++)&\NN\ (5.5.5) \text{\ right\ }\\
\hline
(---+++)&(---)&\NN\ (5.5.5) \text{\ left\ }\\
\hline
(--+-++)&(--+)&\NN\ (4.5.6)\\
\hline
(-+-++-)&(+--)&\NN\ (5.6.4)\\
\hline
(+--+-+)&(-+-)&\NN\ (6.4.5)\\
\hline
(-++-+-)&(+-+)&\NN\ (4.6.5)\\
\hline
(+-+--+)&(-++)&\NN\ (5.4.6)\\
\hline
(++-+--)&(++-)&\NN\ (6.5.4)\\
\hline
\end{tabular}
\end{center}
\caption{Types of North Cask in terms of the signs of the entries of the difference tuple with even index.}
\label{table:types_N_di}
\end{table}

\subsection{South Cask}\label{subsec:south_cask}
In order to study the South Cask of an alcoved polyhedron $\PP$, all we have to do is to turn $\PP$ around, by a
\textbf{Polar Exchange} rotation. We have three of them  $+(12), +(13), +(23)\in\GG_3$ and we can choose any one.

In this section, we  find the relations between the 2--minors and the difference tuples of $A$ and $S$, where the matrix $S$ is such
that the
North (South) Cask of $S$ is the South (North) Cask of $A$, i.e.,
\begin{equation}\label{eqn:S}
\sigma\in\{(12),(13),(23)\} \text{\ and \ }S=\leftidx{^\sigma}A^T.
\end{equation}

The next lemmas  and corollaries
show what happens to matrix entries and
to 2--minors under various matrix transformations. The proofs are straightforward. Recall the group action
$\GG_3\times M_4^{NI}\rightarrow M_4^{NI}$ described in
p.~\pageref{dfn:group_action_1}.

\begin{lem}\label{lem:sigma}Let $\sigma\in\{(12), (13), (23)\}$.
The  conjugate matrix $\leftidx{^{\sigma}}{A}=(c_{ij})$ of matrix  $A\in M_4^{NI}$   satisfies
\begin{equation}\label{eqn:conjugate}
c_{ij}=a_{\sigma(i)\sigma(j)}.\qed
\end{equation}
\end{lem}

\begin{cor}\label{cor:south}Let $\sigma\in\{(12), (13), (23)\}$.
If $A=(a_{ij})\in M_4^{VI}$  and $V=(v_{ij})\in M_4^{VI}$ is the visualization  of \ $\leftidx{^\sigma}A^T$,  then
\begin{equation}\label{eqn:entries_south}
v_{ij}=-a_{\sigma(i)\sigma(j);\sigma(i) 4}, \quad i\neq j\in[4], \quad v_{i4}=a_{\sigma(i) 4}, \quad i\in[3].
\end{equation}
 \end{cor}
 \begin{proof} The entries of $\leftidx{^\sigma}A^T$ are $b_{ij}=a_{\sigma(j)\sigma(i)}$ and the entries of
 $V$ are $v_{ij}=-b_{i4;ij}=-a_{\sigma(i)\sigma(j);\sigma(i)4}$,
 by (\ref{eqn:s_and_a}) and the properties of 2--minors.
 \end{proof}
 Notice:  the minus sign in (\ref{eqn:entries_south}) comes from visualization,  $\sigma$ acts on  subscripts and transposition causes exchange of row and column indices.

\begin{cor}[South Cask]\label{cor:type_South_Cask}
Let $\sigma\in\{(12), (13), (23)\}$. For a maximal alcoved dodecahedron $\DD=\PP(A)$ and  $A\in M_4^{NI}$, the type and edge--lengths of the South Cask
are  determined by the \textbf{signs of three 2--minors} of $A$:
\begin{equation}\label{eqn:2_minors_south}
a_{\sigma(i-1)\sigma(i+1);\sigma(i)4}, \quad i\in [3].
\end{equation}

\end{cor}
\begin{proof}
By lemma \ref{lem:2_minors}, \label{use:lem:2_minors_3} we can assume that $A$ is VI. Then,
if we rotate  the dodecahedron $\DD$ by the Polar Exchange rotation $+\sigma$, it corresponds to looking at  the matrix
$V\in M_4^{VI}$ in corollary \ref{cor:south}. The 2--minors we have to look at are  $v_{i4;i-1,i+1}=v_{i,i-1}-v_{i,i+1}=a_{\sigma(i-1)\sigma(i+1);\sigma(i)4}$,
the last equality by (\ref{eqn:entries_south}) and the cocycle relations (\ref{eqn:cocycle}).  Then  we
apply proposition \ref{prop:type_North_Cask}.
\end{proof}

Now we denote by $G_3$  the  image of the following well--defined group homomorphism
\begin{equation}
\Theta:\GG_3\rightarrow \Z_2 \times \Sigma_6
\end{equation}
which, on a set of generators is given by
\begin{equation}\label{eqn:una}
-\id\mapsto +(14)(25)(36)
\end{equation}
\begin{equation}
+(12)\mapsto -(26)(35)
\end{equation}
\begin{equation}
+(13)\mapsto -(15)(24)
\end{equation}
\begin{equation}
+(23)\mapsto -(13)(46)
\end{equation}
As a consequence, it is easily checked\footnote{Notice the signs in (\ref{eqn:una})--(\ref{eqn:cinco})!} that
\begin{equation}\label{eqn:cinco}
-(23)\mapsto -(16)(25)(34)
\end{equation}

\label{dfn:group_action_3}
The group $G_3$ is an isomorphic copy of $\GG_3$ (easy to check).
Two elements in $G_3$ are of particular relevance: $+(14)(25)(36)$, which is called the
\textbf{three--position shift}, \label{name:3_pos_shift} and
$-(16)(25)(34)$ which is called the \textbf{negated reverse}.\label{name:neg_reverse}

\label{dfn:group_action_2}
We have the \textbf{natural group action} of $G_3\times \R^6\rightarrow\R^6$: for example,
$-(16)(25)(34)\cdot (k_1,k_2,\ldots,k_6)=-(k_6,k_5,\ldots,k_1)$. The isomorphism $\Theta$ induces the  \textbf{group action}
$\GG_3\times \R^6 \rightarrow \R^6$: for example,
$-(23)\cdot (k_1,k_2,\ldots,k_6)=\Theta(-(23))\cdot (k_1,k_2,\ldots,k_6)=-(k_6,k_5,\ldots,k_1)$, by (\ref{eqn:cinco}).

Recall definition \ref{dfn:cant_tuple}.
\begin{lem}[Action of $\GG_3$ on difference tuples]\label{lem:action_on_difference}
If $E$ is a $4\times 4$ perturbation matrix and $d^E$ is the corresponding difference tuple, then
\begin{enumerate}
\item the difference tuple of  matrix $E^T$ is the three--position shift of $d^E$,
\item the difference tuple of  matrix $\leftidx{^{(23)}}E$ is equal to $-(13)(46)\cdot d^E$,
\item the difference tuple of  matrix $\leftidx{^{(23)}}E^T$ is the reverse negated of $d^E$. \label{item:three}
\end{enumerate}
\end{lem}
\begin{proof}
We prove  item \ref{item:three}, the rest being proved similarly. The cant tuple of $E$ is
$c^E=(e_{23},e_{13}, e_{12}, e_{32}, e_{31}, e_{21})$ and the cant tuple of $\leftidx{^{(23)}}E^T$ is
$c=(e_{23},e_{21}, e_{31}, e_{32}, e_{12}, e_{13})$. The difference tuple of $E$ is
$d^E=(e_{13}-e_{23},e_{12}-e_{13}, e_{32}-e_{12}, e_{31}-e_{32}, e_{21}-e_{31}, e_{23}-e_{21})$ and
the difference tuple of $\leftidx{^{(23)}}E^T$ is
 precisely the reverse negated of $d^E$.
\end{proof}

\begin{rem}[Nice Polar Exchange]\label{rem:nice}
The Polar Exchange $+(23)\in\GG_3$ is particularly nice in order to compute the difference tuple of the matrix $S=\leftidx{^{(23)}}A^T$
defined in (\ref{eqn:S}): by item \ref{item:three} in lemma \ref{lem:action_on_difference},
all we have to do is to work on matrix $E$ and follow the crocked closed path  given in definition in p.~\pageref{com:crocked} but \textbf{backwards}.
\end{rem}

\begin{cor}\label{cor:type_South_Cask_VI}
For a maximal alcoved dodecahedron $\DD=\PP(A)$ with $A\in M_4^{VI}$ the type  of the South Cask
is  determined by the \textbf{signs of the entries of the difference tuple of $A$ with odd index}.
\end{cor}
\begin{proof}
We choose the Polar Exchange $+(23)\in\GG_3$ and consider the matrix $S=\leftidx{^{(23)}}A^T$. From  $A=B-E$, it follows that
$S=B'- \leftidx{^{(23)}}E^T$, where both $B$ and $B'=\leftidx{^{(23)}}B^T$ are box matrices. By definition \ref{dfn:perturbation_of_NI}
and  item \ref{item:three} in
lemma \ref{lem:action_on_difference},   the
difference tuples  $d^A=d^E$, $d^S=d^{\leftidx{^{(23)}}E^T}$ satisfy
\begin{equation}\label{eqn:d}
d_2^S=-d_5^A,\quad d_4^S=-d_3^A,\quad d_6^S=-d_1^A
\end{equation}
and the result follows from corollary \ref{cor:type_North_Cask_VI}.
\end{proof}

From corollary \ref{cor:number_types_N} we immediately get the following.
\begin{cor}\label{cor:number_types_S}
The number of  South Cask types of a maximal alcoved dodecahedron
is $8$. \qed
\end{cor}

The next table follows from table \ref{table:types_N_di} and (\ref{eqn:d}). Note that any Polar Exchange  transforms a left pentagon into a right pentagon  (see figure \ref{fig_02_pentagon_left_and_right}).

\begin{table}[ht]
\begin{center}
\begin{tabular}{|>{$}c<{$}|>{$}c<{$}|}
\hline
-\sign(d^A_5,d^A_3,d^A_1)=\sign(d^S_2,d^S_4,d^S_6)&\text{South\  Cask\ type}\\
\hline
(+++)&\SSSS\ (5.5.5) \text{\ right\ }\\
\hline
(---)&\SSSS\ (5.5.5) \text{\ left\ }\\
\hline
(--+)&\SSSS\ (4.5.6)\\
\hline
(+--)&\SSSS\ (5.6.4)\\
\hline
(-+-)&\SSSS\ (6.4.5)\\
\hline
(+-+)&\SSSS\ (4.6.5)\\
\hline
(-++)&\SSSS\ (5.4.6)\\
\hline
(++-)&\SSSS\ (6.5.4)\\
\hline
\end{tabular}
\end{center}
\caption{Types of South Cask in terms of the signs of the entries of the difference tuple with odd index.}
\label{table:types_S_di}
\end{table}

Let us finish this subsection by giving a quick method to find the type of  a South Cask.
Recall that we have rules to label vertices  (see p.~\pageref{dfn:labels} and definitions \ref{dfn:labels_box}
and \ref{dfn:more_labels}). In particular, non--principal vertices  are labeled  $\underline{ij}$ and $\underline{ji}$, with $\underline{ij}\neq\underline{ji}$,
in general.
Notice that, due to the location of columns  given in lemma \ref{lem:geom_interpret},
\label{com:geom_interpret}  vertex  labels in the North Cask follow the same cyclic sequence, no matter the alcoved polyhedron: going counterclockwise around $\NN$, we visit $\underline{1}$,
$\underline{12}$, $\underline{21}$, $\underline{2}$, $\underline{23}$, $\underline{32}$, $\underline{3}$, $\underline{31}$
and $\underline{13}$, in this order. However, the  cyclic vertex sequence in the South Cask, denoted $X$, varies from polyhedron to polyhedron.
For instance $X$ may be equal to $\underline{14},\underline{41},\underline{124},\underline{24},\underline{42},\underline{234},\underline{34},\underline{43},\underline{134}$. There are $8=2^3$   possible sequences $X$,
because $\underline{i4}$ can be preceded or followed by $\underline{4i}$, all $i\in [3]$, as we go around the South Pole
$\SSSS$  clockwise. Each such sequence can be shortened, by omitting the  3--generated vertices; for example,  the former $X$ is shortened to $\underline{14},\underline{41},\underline{24},\underline{42},\underline{34},\underline{43}$.

For a maximal alcoved polyhedra $\PP(A)$, look at the (shortened) sequence $X=X(\PP(A))$.
For $i\in[3]$, it is easy to check that the direction of the segment $s_i$ joining vertices $\underline{4i}$ and $\underline{i4}$
in $\PP(A)$ is
either $u_{i-1}+u_i$ or
$u_{i+1}+u_i$, depending on whether $\underline{i4}$ is preceded or followed by $\underline{4i}$ in $X$.
Since the South Cask type of $\PP(A)$ only depends on the directions of the three segments $s_i$, $i\in[3]$, then it only depends on $X$; see figure \ref{fig_08_types_S}.

The following table shows all the (shortened) $X$ sequences and the corresponding types of South Cask. The sequence $X$ can have zero, one, two or  three inversions.\footnote{We say that $\underline{4i}, \underline{i4}$ gives one inversion, and $\underline{i4}, \underline{4i}$ gives no inversion.}

\begin{table}[ht]
\begin{center}
\begin{tabular}{|>{$}c<{$}|>{$}c<{$}|>{$}c<{$}|}
\hline
\underline{14}\   \underline{41}\   \underline{24}\   \underline{42}\   \underline{34}\   \underline{43}\  & \SSSS (5.5.5) \text{\ left\ }\ &0\\
\hline
\underline{14}\   \underline{41}\   \underline{24}\   \underline{42}\   \underline{43}\   \underline{34}\  & \SSSS (4.6.5)&1\\
\hline
\underline{14}\   \underline{41}\   \underline{42}\   \underline{24}\   \underline{34}\   \underline{43}\  & \SSSS (6.5.4)&1\\
\hline
\underline{41}\   \underline{14}\   \underline{24}\   \underline{42}\   \underline{34}\   \underline{43}\  & \SSSS (5.4.6)&1\\
\hline
\underline{14}\   \underline{41}\  \underline{42}\    \underline{24}\   \underline{43}\   \underline{34}\  & \SSSS (5.6.4)&2\\
\hline
\underline{41}\   \underline{14}\   \underline{24}\   \underline{42}\   \underline{43}\   \underline{34}\  & \SSSS (4.5.6)&2\\
\hline
\underline{41}\   \underline{14}\   \underline{42}\   \underline{24}\   \underline{34}\   \underline{43}\  & \SSSS (6.4.5)&2\\
\hline
\underline{41}\   \underline{14}\   \underline{42}\   \underline{24}\   \underline{43}\   \underline{34}\  & \SSSS (5.5.5)\text{\ right\ }&3\\
\hline
\end{tabular}
\end{center}
\caption{The first column shows the (shortened) sequences $X$, the second column shows the corresponding  South Cask type and  the third column shows the number of inversions.}
\label{table:X}
\end{table}

\begin{figure}[ht]
 \centering
  \includegraphics[width=9cm]{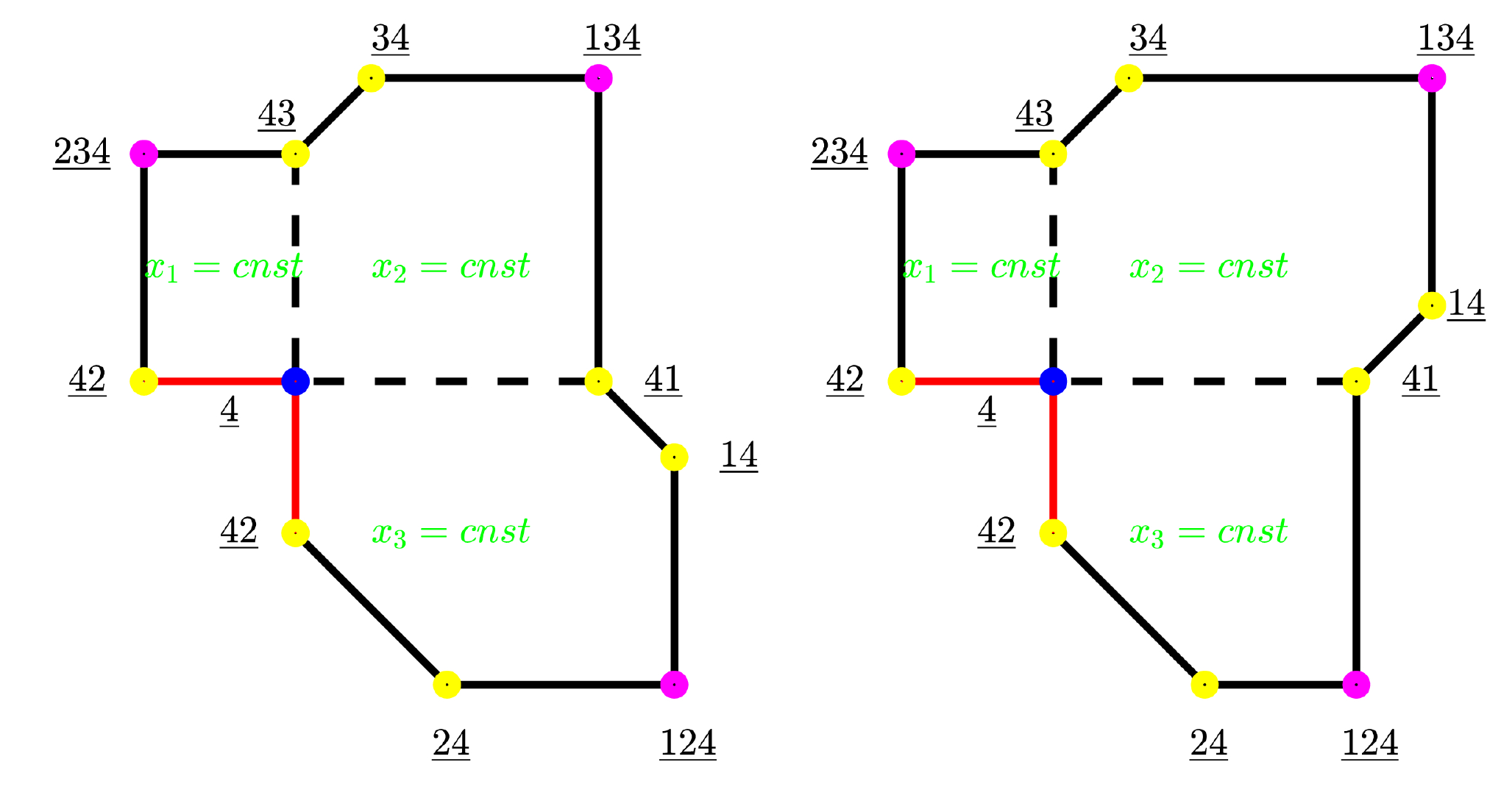}\\
  \caption{Two South Casks. Dashed segments are meant to be folded down. Red segments are meant to be glued together. The direction of
  the segment $s_1$ is $u_1+u_2$ (left)  and $u_1+u_3$ (right).
  The South Cask is $(4.5.6)$  (left) and  $(4.6.5)$  (right). The shortened sequence $X$ is $\underline{41},\underline{14},\underline{24},\underline{42},\underline{43},\underline{34}$ (left) and  $\underline{14},\underline{41},\underline{24},\underline{42},\underline{43},\underline{34}$ (right), going clockwise.}\label{fig_08_types_S}
\end{figure}

\subsection{The effect of Polar Exchange  on Cask types}
For $\sigma\in\{(12),(13),(23)\}$,
the Polar Exchange $+\sigma\in\GG_3$  transforms  the facet $x_i=1$ into  the facet $x_j=-1$ and the
facet $x_k=1$ into  the facet $x_k=-1$ in the cube $\QQ$, when $\sigma=(ij)$. Note that any Polar Exchange  transforms a left pentagon into a right pentagon. The next corollary  follows.

\begin{cor}\label{cor:effects_of_psi}
For $A\in M_4^{NI}$ and $\sigma\in\{(12),(13),(23)\}$,
\begin{enumerate}
\item if  $\NN$ is $(5.5.5)$ right in $\PP(A)$, then  $\SSSS$ is $(5.5.5)$ left in  the rotated polyhedron $+\sigma(\PP(A))$, \label{item:1}
\item if  $\NN$ is $(p.q.r)$ in $\PP(A)$, then  $\SSSS$ is $(\sigma(p).\sigma(q).\sigma(r))$ in  the rotated polyhedron $+\sigma(\PP(A))$. \label{item:2} \qed
\end{enumerate}
\end{cor}

\subsection{North and South Casks}
In this section we bring together previous results and give an example on how to compute the Cask types from the difference tuple.
The following comes from corollaries \ref{cor:type_North_Cask_VI}, \ref{cor:type_South_Cask_VI} and definition \ref{dfn:perturbation_of_NI}.

\begin{cor}[Polar types determined by $\sign(d)$]\label{cor:type_global_2}
For a maximal alcoved dodecahedron $\PP(A)=\DD$ and  $A\in M_4^{NI}$, the types of the North and South Casks
are determined by  the signs  of the difference tuple $d\in\R^6$.
More precisely, $\sign(d_I)$, with $I$ even,  determine the type of the North Cask, and
$\sign(d_I)$, with $I$ odd, determine the type of the South Cask. \qed
\end{cor}

\begin{ex}\label{ex:example}
Consider the VI matrix  $A$ in example 2 in p.~\pageref{ex:QE2}. The box is a cube of edge--length 8  and $A=B-E$, where $E=\left(\begin{array}{rrrr}
0&-4&-3&0\\-5&0&-2&0\\-6&-5&0&0\\0&0&0&0\\
\end{array} \right)$, $c=(-2,-3,-4,-5,-6,-5)$, $d=(-1,-1,-1,-1,1,3)$ are the cant and difference tuples of $A$.
 To determine the  North Cask type of   $\PP(A)$ we use
 $d_2<0$, $d_4<0$, $d_6>0$,
whence the North Cask type is   $(4.5.6)$, by line 3 in table \ref{table:types_N_di}. To determine the  South Cask type of   $\PP(A)$ we use
$-\sign(d_5,d_3,d_1)=(-++)$,  whence the South Cask type is   $(5.4.6)$, by line 7 in table \ref{table:types_S_di}.
\end{ex}

\begin{cor}[Two impossible cases]\label{cor:impossible}
For a matrix $A\in M_4^{VI}$, it is impossible to have $\NN$  $(5.5.5)$ right and $\SSSS$  $(5.5.5)$ right. Similarly, it is impossible
to have $\NN$  $(5.5.5)$ left and $\SSSS$  $(5.5.5)$ left.
\end{cor}
\begin{proof}
We only prove the first statement. Let $c$ and  $d$ be the cant  and difference tuples.
Now,  $\NN$  $(5.5.5)$ right is equivalent to  $d_I>0$, $I$  even, and $\SSSS$  $(5.5.5)$ right is equivalent to $d_I>0$, $I$ odd,
(using  the first lines in tables \ref{table:types_N_di} and \ref{table:types_S_di}). This  contradicts $\sum_I d_I=\sum_i c_i-\sum_i c_i=0$.
 \end{proof}
The corollary above was proved in  \cite{Jimenez_Puente}, Theorem 14, with a longer proof.
Warning: notation disagreement:   $\SSSS  (5.5.5)$ left in \cite{Jimenez_Puente} is denoted $\SSSS  (5.5.5)$ right here.
\subsection{Equatorial Belt}\label{subsec:belt}
A box is  not  maximal with respect to the $f$--vector. It is the union of a North and a South Casks.
In this section we present an alcoved  polyhedron $\PP$  as the union of a North and a South Casks and an Equatorial Belt.
We show that this Belt is uniquely determined by the Polar Casks.

\begin{dfn}[Equatorial Belt]\label{dfn:Belt}
The Equatorial Belt of an alcoved polyhedron is the set of facets not belonging to the Polar Casks.
\end{dfn}

\begin{figure}[ht]
\centering
\includegraphics[width=10cm]{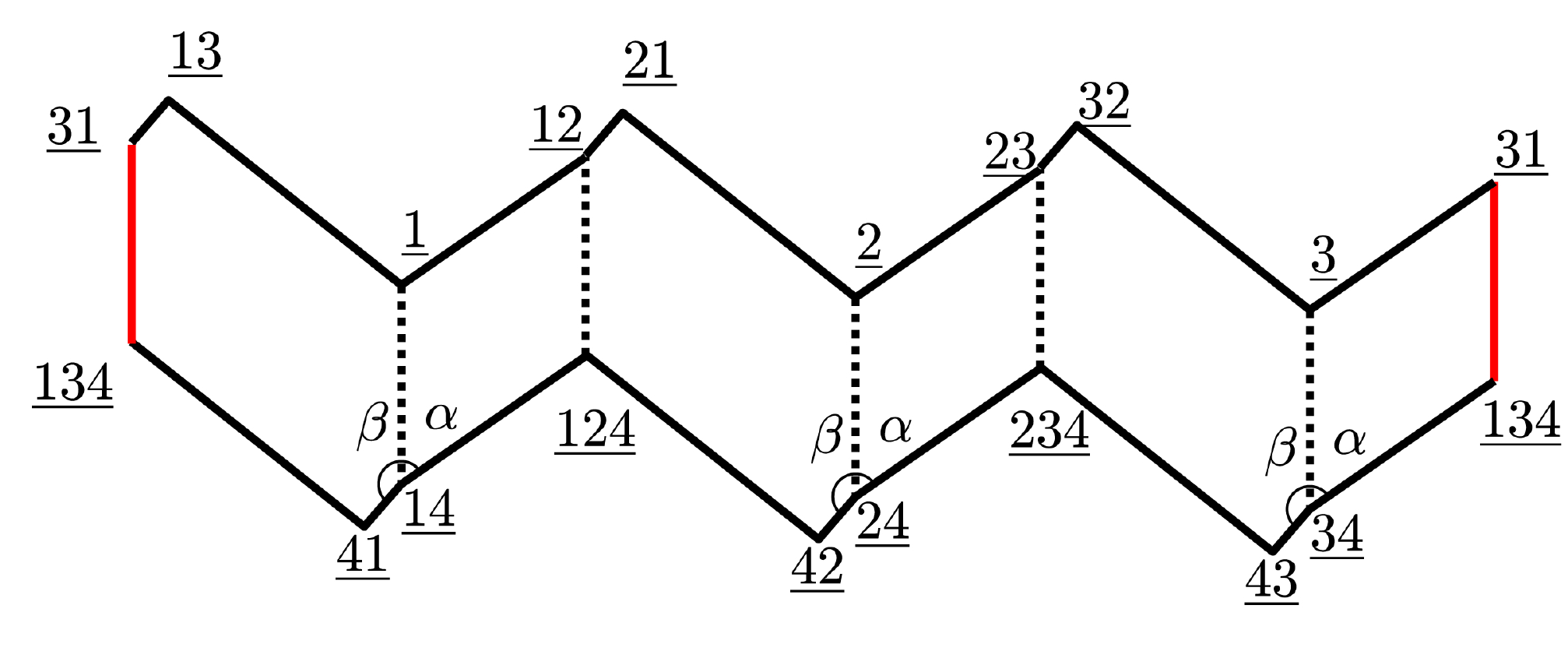}
\caption{Equatorial belt of example 6 in p.~\pageref{ex:QE6}. Dotted segments are meant to be folded down. Red segments are meant to be glued together. The vertical segments have direction $u_1+u_2+u_3$ and  $\alpha=\arccos(\sqrt{1/3})\simeq 54^\circ 44'$, $\beta=\arccos(\sqrt{2/3})\simeq 144^\circ  44'$ (see angles in p.~\pageref{com:angles}).}\label{fig_09_equatorial_belt_QE6}
\end{figure}

The Equatorial Belt of a box is trivial, while the Belt of the alcoved polyhedron in example \ref{ex:one_cant}
consist of just one 4--gon facet.

Let $(u_1,u_2,u_3)$ be the canonical basis in $\R^3$.
In a maximal alcoved  dodecahedron $\DD\subset \R^3$, the facets
not belonging to the Polar Casks contain some edge in the direction $u_1+u_2+u_3$, and conversely.
In order to get the Belt facets, we must find the edges with this direction. Clearly, three such edges are
$\underline{1}-\underline{14}$, $\underline{2}-\underline{24}$
and $\underline{3}-\underline{34}$. These edges determine six facets in $\DD$, which are the  Belt facets of $\DD$.

\begin{rem}\label{rem:Belt}
The Equatorial Belt of a maximal alcoved  dodecahedron $\DD\subset \R^3$ is uniquely determined by the Polar Casks.
\end{rem}

\begin{nota}\label{nota:Belt}
Let $\DD$ be a maximal alcoved dodecahedron.
In $\DD$, each Belt facet $F_j$ is the result of canting a
cantable edge $\ell_j$ in the bounding box $\BB$, $j\in [6]$ (cf. figure \ref{fig_03_three_cubes}).
The Belt is denoted $EB=(q_1,q_2,\ldots,q_6)$,  where  $q_j$ is the number of edges of $F_j$. By item \ref{dfn:p_vector} in section \ref{sec:maximality}, we know that  $q_j\in\{4,5,6\}$.
\end{nota}

\section{Families, group actions and quasi--Euclidean classification}\label{sec:qe_classif}
In this section,
we present a classification of alcoved polyhedra that encompasses three sorts of equivalence.
First, we want  all boxes to be equivalent. This equivalence is, so to say, affine. Second,  for two alcoved polyhedra $\PP$ and
$\PP'$ having the cube $\QQ$ as bounding box, if there exists a symmetry in the group $\GG_3$ taking $\PP$ to $\PP'$, then we want $\PP$
and $\PP'$ to be equivalent.  This equivalence is, so to say, Euclidean. Third, if there is a sufficiently small perturbation taking
$\PP$ to $\PP'$, then we also want $\PP$ and $\PP'$ to be equivalent. This equivalence is, so to say, topological.

These ideas require that we take \textbf{four steps in order to define the desired equivalence relation}.
The second and third steps define equivalence relations,  but only for polyhedra bounded by  the unit cube $\QQ$.

\medskip

For a given  box $\BB\subset \R^3$, let $\PPP_{\BB}$ be the \textbf{family}  of all
\textbf{alcoved convex  polytopes whose box is $\BB$} and let $\DDD_{\BB}\subset \PPP_{\BB}$  be the
\textbf{family}  of all \textbf{maximal alcoved dodecahedra whose box is $\BB$}.
Let $\PPP$ be the \textbf{family}  of all \textbf{alcoved polyhedra} and $\DDD$ the
\textbf{family}  of
all \textbf{maximal alcoved dodecahedra}.

\begin{dfn}[First step: Equivalence of boxes]\label{dfn:eq_1}
All boxes in $\R^3$ are equivalent.
\end{dfn}

\begin{dfn}[Second step: Euclidean classes in $\PPP_{\QQ}$ and in $\DDD_{\QQ}$]\label{dfn:eq_2}
The group $\GG_3$ acts on $\PPP_{\QQ}$ (resp. $\DDD_{\QQ}$) and an \textbf{Euclidean class} in $\PPP_{\QQ}$  (resp. in $\DDD_{\QQ}$)  is
a $\GG_3$--orbit in $\PPP_{\QQ}$ (resp. $\DDD_{\QQ}$).
\end{dfn}

Since the orbits of  $\GG_3$ have, at most 12 elements, the Euclidean classes defined above are many.
Inside each of them, Euclidean angles and distances (either Euclidean or tropical) are preserved.

\begin{dfn}[Third step: Equivalence by small perturbation]\label{dfn:eq_3}
Consider the  matrix $\leftidx{^D}Q\in M_4^{VI}$ in example \ref{ex:unit_cube}, the visualized matrix of the unit cube $\QQ$. For $i,j\in [3]$, $i\neq j$ and
real number $\epsilon$, consider
the \textbf{elementary matrices} $P_{ij,\epsilon}=(p_{kl})$ with $p_{ij}=\epsilon$ and $p_{kl}=0$ whenever $(k,l)\neq(i,j)$. For
$A\in M_4^{NI}$, define $A_{ij,\epsilon}:=\leftidx{^D}Q+P_{ij,\epsilon}$. For $\epsilon\in[0,2]$ we have $A_{ij,\epsilon}\in M_4^{NI}$
(easy to prove) and
for $\epsilon,\epsilon'\in(0,2)$ we define $\PP(A_{ij,\epsilon})$ and $\PP(A_{ij,\epsilon'})$ to be equivalent.
\end{dfn}

The extreme cases are    $\epsilon=2$, providing  a half--cube and  $\epsilon=0$ providing the whole cube, with
matrix $A_{23,\epsilon}=\leftidx{^D}Q+P_{23,\epsilon}=\begin{pmatrix}
0&-2&-2&-2\\-2&0&-2+\epsilon&-2\\-2&-2&0&-2\\0&0&0&0\\
\end{pmatrix}$. The polyhedra $\PP(A_{23,2})$ and $\PP(A_{23,0})$ are not quasi--Euclidean equivalent to $\PP(A_{ij,\epsilon})$, for $\epsilon\in(0,2)$  (see figure \ref{fig_10_non_equivalent}).

\begin{figure}[ht]
 \centering
  \includegraphics[width=9cm]{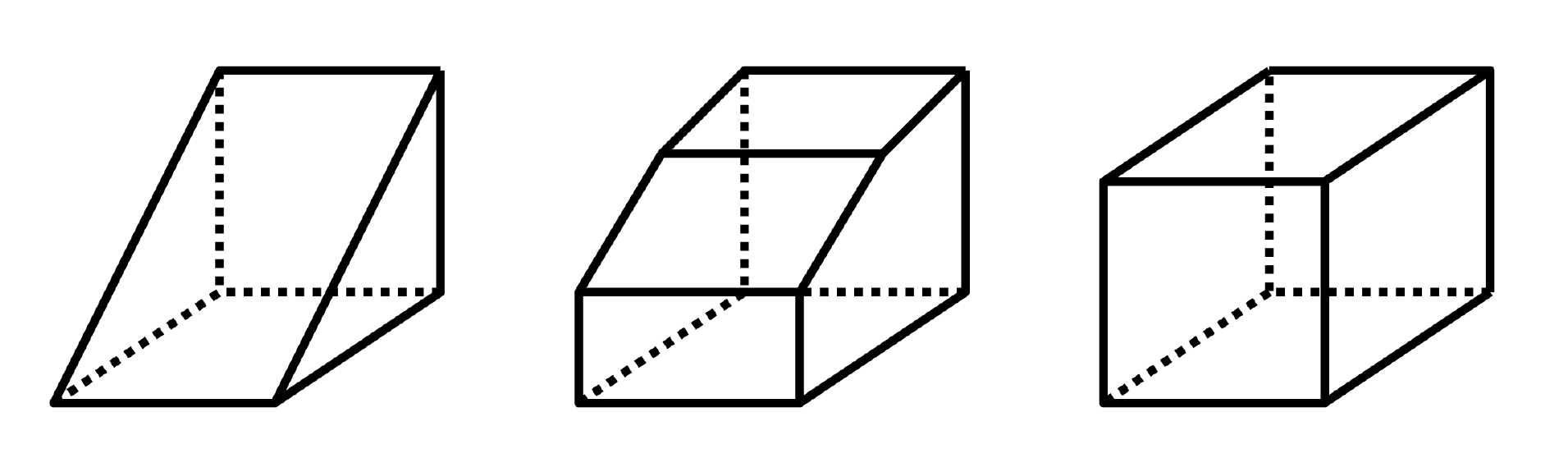}\\
  \caption{The three alcoved polyhedra shown are not quasi--Euclidean equivalent: they are all perturbations of the unit cube $\QQ$. }\label{fig_10_non_equivalent}
\end{figure}

Next, we define equivalence between polyhedra bounded by different  bounding boxes, putting together the  three definitions above.

\begin{dfn}[Fourth step: Quasi--Euclidean classes]\label{dfn:eq_4}
Let $\PP$ and $\PP'\subset \R^3$ be alcoved polyhedra, $\BB$ and $\BB'$ be the corresponding bounding boxes.
Let $f_{\BB,\BB'}:\R^3\rightarrow\R^3$ be the unique affine linear map taking generators
to generators, i.e.,  $f_{\BB,\BB'}(\underline{j})=\underline{j}$, for $j\in[4]$.
Let
$\PP_{\QQ}$ and $\PP'_{\QQ}$ be the images
under $f_{\BB,\QQ}$ and $f_{\BB',\QQ}$, respectively. If $\PP_{\QQ}$ and $\PP'_{\QQ}$ are equivalent
(according to definitions \ref{dfn:eq_2} and \ref{dfn:eq_3}), then we define $\PP$ and $\PP'$  to be
quasi--Euclidean equivalent.
It can be checked that this is an
equivalence relation in $\PPP$ (resp. $\DDD$).
\end{dfn}

For instance,
any two boxes with exactly
one sufficiently small cant are quasi--Euclidean equivalent.

Finally, we are ready to give the \textbf{proof to theorem \ref{thm:eight_classes}}.
\begin{proof}\label{proof:main_thm}
It is enough to find exactly eight different quasi--Euclidean classes in the family $\DDD_{\QQ}$, where $\QQ$ is the unit cube. In order
to do so, we combine North and  South
Casks,\footnote{By remark \ref{rem:Belt}, the Equatorial Belt is not to be taken int account.} getting $8^2=64$ possible cases.
By corollary \ref{cor:impossible}, two cases are ruled
out,  and we are left with $62$ cases.

Consider  $A\in M_4^{NI}$.
In corollary \ref{cor:type_global_2} we have seen that the types of the North and South Casks of $\PP(A)$
are determined by  the signs  of the difference tuple $d\in\R^6$ of $A$.

Consider the 62--element set $T$ consisting of all non--constant ordered tuples in $\{+,-\}^6$.  Consider the
\textbf{natural group action} \label{dfn:group_action_4} $G_3\times T\rightarrow T$ (the group $G_3$ was defined in p.~\pageref{dfn:group_action_3}). For example,
$-(16)(25)(34)\cdot (++++--)=-(--++++)=(++----)$. With a simple computer program, it is easily checked  that this action has exactly
8 orbits. Each orbit corresponds to one quasi--Euclidean class and the proof of theorem \ref{thm:eight_classes} is complete.
\end{proof}

\begin{rem}
\begin{enumerate}
\item The 8 examples in subsection \ref{subsec:examples} belong to different classes.
The cardinality of the corresponding orbits is 6, 12, 6, 6, 6, 2, 12 and 12, adding up to 62.
\item The reader should \emph{not} believe that the action  $G_3\times T\rightarrow T$ used in the proof above  equals the action on $T$ generated by one--shift, reverse and negate.
Indeed, the subgroup  of $\Z_2\times \Sigma_6$ generated by one--shift, reverse and negate is strictly larger than $G_3$.
\end{enumerate}
\end{rem}

\medskip

\begin{rem}\label{rem:relation_with_Jimenez_Puente}
In \cite{Jimenez_Puente}, the combinatorial classification of maximal alcoved dodecahedra is achieved attending to three
vectors: $t$, $p$ and $h$ (which are not independent and can attain only a few integer values).  Vectors $t, p, h$ can be readily
obtained from   the Cask
types and the Equatorial Belt. Indeed, the $p$--vector $(p_4,p_5,p_6)$ (see item \ref{dfn:p_vector} in section \ref{sec:maximality})
is obtained
by counting  4--gons, 5--gons and 6--gons. For instance, in example 3 in p. \pageref{ex:QE3} we have $\NN (4.5.6)$,
$\SSSS (6.5.4)$  and
$EB=(5,4,5,5,6,5)$, whence we have a total of three 4--gons, six 5--gons and three 6--gons, so $p=(3,6,3)$.

The \emph{hexagons--vector} or \emph{h--vector} \label{dfn:h_vector} was defined in \cite{Jimenez_Puente}.\footnote{The
$h$--vector considered here has nothing to do with the $h$--vector found in the literature on $f$--vectors.}
 It is
$h=(h_1,h_2,h_3,h_4)\in (\N\cup\{0\})^4$, where  $h_j$ denotes the number of maximal families consisting of $j$ pairwise adjacent hexagons. Clearly $h_1+2h_2+3h_3+4h_4=p_6$.  For instance, in example 3 in p. \pageref{ex:QE3} we have three 6--gons not touching each other, whence $h=(3,0,0,0)$.

The \emph{type--vector} or \emph{t--vector}  \label{dfn:t_vector} was defined in \cite{Jimenez_Puente}. It is  $t=(t_1,t_2,t_3)\in (\N\cup\{0\})^3$,  where $t_j$ is the number of  edges in the direction $u_{j-1}+u_{j+1}$, indices modulo 3 (recall item \ref{item:last} in remark \ref{rem:odds_and_ends}). For instance, in example 3 in p. \pageref{ex:QE3} we have two edges in the direction $u_1+u_2$ (resp. $u_2+u_3$) (resp. $u_3+u_1$) whence $t=(2,2,2)$.
Table \ref{table:qe_vs_combinatorial} shows the relation between the two classifications. The combinatorial classification is coarser than the quasi--Euclidean one.
\end{rem}

\begin{table}[ht]
\begin{center}
\begin{tabular}{|>{$}l<{$}||>{$}l<{$}| >{$}l<{$}| >{$}l<{$}| >{$}l<{$}| >{$}l<{$}|}
\hline
\text{QE\ class}&QE 1&QE 2&QE 3&QE 4&QE 5\\
p&(4,4,4)&(3,6,3)&(3,6,3)&(4,4,4)&(4,4,4)\\
h&(0,0,0,1)&(1,1,0,0)&(3,0,0,0)&(0,2,0,0)&(0,0,0,1)\\
t&(0,2,4)&(1,3,2)&(2,2,2)&(0,3,3)&(11,1,4)\\
\text{Comb.\ class} &6&2&1&4&5\\
\hline
\text{QE\ class}&&QE 6&QE 7&QE 8&\\
p&&(3,6,3)&(2,8,2)&(3,6,3)&\\
h&&(3,0,0,0)&(0,1,0,0)&(1,1,0,0)&\\
t&&(2,2,2)&(1,2,3)&(2,1,3)&\\
\text{Comb.\ class}&&1&3&2&\\
\hline
\end{tabular}
\end{center}
\caption{This table shows  the QE class, the $p$, $h$  and $t$ vectors
and the Combinatorial class it corresponds to, according to
\cite{Jimenez_Puente}.}
\label{table:qe_vs_combinatorial}
\end{table}

\subsection{Examples}\label{subsec:examples}
We give examples for all eight  quasi--Euclidean Classes, indicating a VI matrix, the difference tuple, the North and South Casks, the  Equatorial Belt,
and the $p$--vector  (defined in  item \ref{dfn:p_vector}, section \ref{sec:maximality}). The bounding box is, in every case, the cube of edge--length 8 (this can be seen in the matrix last column).

\begin{enumerate}
\item $\left(\begin{array}{rrrr}0&-4&-5&-8\\-3&0&-6&-8\\-4&-5&0&-8\\0&0&0&0\end{array}\right)$, $d=(-1,-1,1,-1,-1,3)$, $\NN (4.5.6)$, $\SSSS (4.5.6)$, $EB=(4,5,6,4,5,6)$, $p=(4,4,4)$ is an example for $QE1$

\item $\left(\begin{array}{rrrr}0&-4&-5&-8\\-3&0&-6&-8\\-2&-3&0&-8\\0&0&0&0\end{array}\right)$, $d=(-1,-1,-1,-1,1,3)$, $\NN (4.5.6)$, $\SSSS (5.4.6)$, $EB=(4,5,5,5,6,5)$, $p=(3,6,3)$ is an example for $QE2$ \label{ex:QE2}

\item $\left(\begin{array}{rrrr}0&-5&-6&-8\\-4&0&-5&-8\\-3&-4&0&-8\\0&0&0&0\end{array}\right)$, $d=(1,-1,-1,-1,1,1)$, $\NN (4.5.6)$, $\SSSS (6.5.4)$, $EB=(5,4,5,5,6,5)$, $p=(3,6,3)$ is an example for $QE3$ \label{ex:QE3}

\item $\left(\begin{array}{rrrr}0&-4&-5&-8\\-5&0&-6&-8\\-4&-5&0&-8\\0&0&0&0\end{array}\right)$, $d=(-1,-1,1,-1,1,1)$, $\NN (4.5.6)$, $\SSSS (4.6.5)$, $EB=(4,5,6,4,6,5)$, $p=(4,4,4)$ is an example for $QE4$

\item $\left(\begin{array}{rrrr}0&-5&-6&-8\\-6&0&-5&-8\\-5&-4&0&-8\\0&0&0&0\end{array}\right)$, $d=(1,-1,-1,1,1,-1)$, $\NN (4.5.6)$, $\SSSS (5.4.6)$, $EB=(5,4,6,4,5,6)$, $p=(4,4,4)$ is an example for $QE5$

\item $\left(\begin{array}{rrrr}0&-6&-5&-8\\-5&0&-6&-8\\-6&-5&0&-8\\0&0&0&0\end{array}\right)$, $d=(-1,1,-1,1,-1,1)$, $\NN (5.5.5)$ right, $\SSSS (5.5.5)$ left, $EB=(4,6,4,6,4,6)$, $p=(3,6,3)$ is an example for $QE6$ \label{ex:QE6}

\item $\left(\begin{array}{rrrr}0&-5&-6&-8\\-2&0&-7&-8\\-3&-4&0&-8\\0&0&0&0\end{array}\right)$, $d=(-1,-1,-1,-1,-1,5)$, $\NN (4.5.6)$, $\SSSS (5.5.5)$ left, $EB=(4,5,5,5,5,6)$, $p=(2,8,2)$ is an example for $QE7$

\item $\left(\begin{array}{rrrr}0&-4&-5&-8\\-3&0&-6&-8\\-4&-3&0&-8\\0&0&0&0\end{array}\right)$, $d=(-1,-1,-1,1,-1,3)$, $\NN (5.4.6)$, $\SSSS (5.5.5)$ left, $EB=(4,5,5,6,4,6)$, $p=(3,6,3)$ is an example for $QE8$.
\end{enumerate}

\section{Final remarks}\label{sec:final}
\subsection{Partial answers to two open questions}\label{subsec:partial_ans}
Below we give answers
in the case of maximal alcoved polyhedra. The case of arbitrary alcoved polyhedra follows from here  (by degeneration, i.e., collapsing of vertices).

\textbf{Question 1:} \textbf{How many distinct combinatorial types of  polyhedra belong to each
facet vector?} (question posed in \cite{Senechal} p.~198).
For  maximal alcoved dodecahedra, the  facet vectors
are  $p=(p_4,p_5,p_6)=(2,8,2)$, $(3,6,3)$  and  $(4,4,4)$ (see item \ref{dfn:p_vector} in section \ref{sec:maximality}) and the
number of combinatorial types are  $1,2$ and $3$  resp. (so  that $1+2+3=6$). This  is readily checked in   table \ref{table:qe_vs_combinatorial}, and follows from \cite{Jimenez_Puente}.

\textbf{Question 2} is
\textbf{D\"urer's Problem}, also known as \textbf{Shephard's conjecture}
(1975): \textbf{Does a net exist?} (question posed in \cite{ORourke}) The answer is yes, for every maximal alcoved dodecahedron $\DD$.
Indeed, the Equatorial Belt  can be unfolded (cutting one $u_1+u_2+u_3$ edge) to a planar layout (cf. figure \ref{fig_09_equatorial_belt_QE6}).
Besides, planar layouts of North and South Casks  exist, and  the three layouts can be glued together. Then one has to check that no
overlaps occur, for which one uses that  facial angles of $\DD$ are restricted, belonging to the following finite set  (in degrees, rounded to minutes):
$$\{90^\circ, 45^\circ, 135^\circ, 60^\circ, 120^\circ,
54^\circ 44', 125^\circ  16',  35^\circ 16',144^\circ  44'\}.$$ The latter set arises from the set \label{com:angles}
$$\{\arccos (0),\arccos(\sqrt{1/2}),\arccos (1/2),\arccos(\sqrt{1/3}), \arccos(\sqrt{2/3})\},$$ which are the angles  formed by pairs of
vectors in
$\{u_i, u_i+u_j, u_1+u_2+u_3\}$.  Non overlapping is checked class by class, in the quasi--Euclidean classification.

\subsection{Alcoved polyhedra occur in nature}\label{subsec:nature}
We have found the combinatorial types of certain alcoved polyhedra in two instances.
First, in Schlegel diagrams  of  soap bubbles assembled in a transparent container (see \cite{Fleck} p.~181, fig. 9.13). There,
three  classes of maximal dodecahedra are shown (items 3, 4 and 5 in fig. 9.13). Second, natural gas hydrates (also called clathrates) are
solid compounds of small gas molecules and water. The H structure of clathrates (shown  in figure in p.~354, bottom, in \cite{Sloan}) has the combinatorial type of a maximal alcoved dodecahedron.

\section*{Acknowledgments}\label{sec:acknowledgments} I am deeply grateful to the referee for careful reading and interest. His/her suggestions and patience have been  a great help  to bring this paper to light.
I also thank my friend  P.L. Claver\'{\i}a for producing 3--D  models of many alcoved dodecahedra   and for checking many computations.


\end{document}